\documentclass[11pt]{article}
\usepackage{amsmath,amssymb, enumerate, fullpage}
\usepackage{graphicx,color}
\pdfoutput=1
\newtheorem{theorem}{Theorem}
\newtheorem{lemma}{Lemma}
\newtheorem{assumption}{Assumption}
\newtheorem{remark}{Remark}
\newtheorem{definition}{Definition}

\usepackage[tight,footnotesize]{subfigure}
\usepackage{epstopdf}
\usepackage{epsfig}

\newcommand{\tr}{^{\sf T}}

\begin{document}
\title{Convergence Analysis of A Proximal Linearized ADMM Algorithm  for Nonconvex Nonsmooth Optimization
 }

\author{Maryam Yashtini }
\author{
 Maryam Yashtini 
\thanks{my496@georgetown.edu,
Georgetown University, 
Department of Mathematics and Statistics
327A St. Mary's Hall
37th and O Streets, N.W., Washington D.C. 20057
Phone: (202) 687-6214
Fax: (202) 687.6067
}
}\maketitle

\abstract{ 
In this paper, we consider a proximal linearized alternating direction method of multipliers 
(PL-ADMM)   for solving linearly constrained nonconvex and possibly nonsmooth optimization 
problems. The algorithm is generalized by using variable metric proximal terms in the primal updates 
and an over-relaxation stepsize in the multiplier update.  
We prove that the sequence generated by this method is bounded and its limit points 
are critical points. Under the powerful Kurdyka-{\L ojasiewicz} properties
we prove that the sequence has a finite length thus converges, and we drive its convergence rates.
}

\section{Introduction}
\subsection{History}
The alternating direction method of multipliers   (ADMM) 
\cite{DistribOpt,DengYin16,EcksteinBertsekas92,GabayMercier76,hnyz15}
is one of the most successful mathematical methodologies to solve linearly 
constrained optimization problems. 
ADMM is closely related to the Douglas-Rachford \cite{DouglasRachford56}
and Peachman-Rachford \cite{PeacemanRachford55} operator splitting methods that date back to the
1950s. One of the significant advantages of using the ADMM to solve structured problems is that it reduces the complexity of original problem by breaking it into several simpler minimization subproblems
that each can be solved independently. 
The approximate local solutions to these subproblems are then
coordinated to find a global solution to the original problem. 
Some important applications of ADMM and its variants include machine learning, statistics,
compressive sensing, image and signal processing,
sparse and low-rank approximations, see 
\cite{RaymondChan15,chy13,hnyz15,AdmmConsensus14,WT10,YangZhangYin2010,yhcy12, EdgeColorYKZ19,YKssvm15,MaryamKangSIIM}
and the surveys \cite{DistribOpt, EcksteinYao15}.

Theoretical analysis of the ADMM  has extensively studied in the context of {\it convex optimization}  
\cite{boley13,DengYin16,EcksteinBertsekas92,GabayMercier76,hyz13, LionsMercier79}.
Boley \cite{boley13} studied the local linear convergence for solving quadratic and linear
programs. 
Deng and Yin \cite{DengYin16} studied the convergence rate of a
a general ADMM  method in which a proximal term 
 was added to each subproblem.
 Eckstein and Bertsekas \cite{EcksteinBertsekas92} proved the
linear convergence for solving linear programs, which
depends on a bound on the largest iterate in the course of the ADMM algorithm.
Hager, Yashtini, and Zhang \cite{hyz13} established the ergodic convergence rate
of a proximal linearized ADMM method \cite{chy13}, in which the proximal parameter updates
 through a backtracking line search strategy.
Lions and Mercier  \cite{LionsMercier79} showed that the Douglas--Rachford operator 
splitting method converges linearly under the assumption that some involved monotone
operators are both coercive and Lipschitz. 
In  \cite{GoldfarbMa12} the authors established the  $O(1/k)$ and $O(1/k^2)$ convergence rates of 
a Jacobi version of ADMM (MSA) and its accelerated variant (FaMSA) respectively,
when both functions in the objective are convex smooth and have Lipschitz continuous gradients. 
Similar results also established in \cite{GoldfarbMaScheinberg13} for a Gauss-Seidel variant of ADMM
(ALM) and its accelerated variant (AccALM), which requires only one of the two functions to be smooth with 
Lipschitz continuous gradients. 
Based on the variational inequality, He and Yuan established the $O(1/k)$ convergence rate of ADMM for the values \cite{HeYuan12} and the sequence \cite{HeYuan14}.
Goldstein and Donoghue  \cite{GoldsteinDonoghue14} proved the 
$O(1/k)$ rate for the ADMM and $O(1/k^2)$ rate for its accelerated variant (Fast ADMM)
based on the dual objective function under strong convexity assumptions.
Davis and Yin  \cite{DavisYin17} showed that the linear and sublinear rates of ADMM can be 
obtained as an extension of Douglas-Rachford splitting method in the sense of fixed-point 
residual and objective error.
The R-linear convergence rate of ADMM under certain error bound condition
proved in \cite{HonLuo17}. 

The convergence of multi-block ADMM for minimizing the sum of more than two 
functions requires strong convexity assumptions on some
or all functions in the objective 
 \cite{CaiHanYuan14,ChenShenYou15,DengYin16,HanYuan12,LinMaZhang15,LinMaZhang2018}. 
 Without strong convexity assumptions, Chen et al. \cite{ChenHeYuanYe16}  showed that the multi-block ADMM 
is not necessarily convergent unless there exists at least two orthogonal coefficient matrices. 
Deng et al. \cite{DengLaiPengYin17} showed that the multi-block Jacobi ADMM 
converges with the $o(1/k)$ rate when all coefficient matrices are mutually near-orthogonal and have full 
column rank or proximal terms are added to the subproblems.
In \cite{HeTaoYuan12,HeTaoXuYuan10} the authors combined  the ADMM with either forward or backward 
substitution procedure and they proved the convergence results from contraction perspectives.
He, Hou, and Yuan \cite{HeTaoYuan17} showed that the local linear convergence rate is provable
 if certain standard error bound condition is assumed. 
 He et al. \cite{HeHouYuan15} combined the Jacobi ADMM with a relaxation step
and proved the convergence and $O(1/k)$ rate in both ergodic and non-ergodic senses. 

The ADMM algorithm generally fails to solve {\it nonconvex possibly nonsmooth} optimization problems.
However, its great performance on some practical applications such as  phase retrieval \cite{Wen12},
distributed clustering \cite{Forero11}, 
sparse zero variance discriminant analysis \cite{AmesHong15},
matrix separation \cite{matrixsep14},
imaging \cite{MaryamKangSIIM,YKssvm15,EdgeColorYKZ19},
sparse feedback control \cite{sparsefeedback} 
has encouraged researchers to study underlying conditions for which 
nonconvex nonsmooth ADMM converges. 
Wang and Yin \cite{WangYin19}  established the convergence of nonconvex nonsmooth multi-block ADMM  
for both separable and non-separable objective functions.
For the separable objective function they assumed that majority of functions in the objective are 
prox-regular. When the functions are either smooth and nonconvex or convex and nonsmooth,
Hong et al. \cite{HongLuoRazaviyayn16} proved the convergence of
the ADMM, provided that the penalty parameter in the augmented Lagrangian is chosen to be sufficiently large.
Wang et al.  \cite{WangCaoXu18}  analyzed the so-called Bregman ADMM under the assumption that the objective function is subanalytic, and included the standard ADMM as a special case.
 In \cite{LiuShenGu19} the authors studied the convergence of a full linearized
ADMM, given that one of the coefficient matrices is full column rank and the 
penalty parameter in the augmented Lagrangian is sufficiently large.
The work \cite{JiangLinMaZhang19} considered a variant of multi-block proximal ADMM, where 
the proximal ADMM updates can be implemented for all the block variables 
except for the last block, for which either a gradient step or a majorization–minimization step is implemented. 
The authors analyzed its iteration complexity. 
In \cite{GuoHanWu17} Guo et al. considered a two-block ADMM to minimize a sum of two nonconvex functions with linear constraints and they proved the convergence by assuming that the generated sequence is bounded.
Under some assumptions on the coefficient matrices,
Yang, Pong, and Chen  \cite{YangPongChen17} studied a three-block ADMM to solve a special class of nonconvex and nonsmooth problems with applications to background/foreground extraction.
Li and Pong \cite{LiPong15} proved the global convergence of  two-block ADMM 
under assumption that one of the functions in the objective is twice differentiable and has bounded Hessian.
Under Kurdyka-\L ojasiewicz  property, 
Bot and Nguyen \cite{ProxADMMnoncvx19} and Yashtini \cite{yashtini2020-GADMM}  
established the convergence and convergence rates of proximal variants of ADMM. 
The iteration complexity of two classes of fully and partially linearized multi-block ADMM
with the choice of some relaxation parameter in the multiplier update established in \cite{MeloMonteiro17}.
In \cite{ThemelisPatrinos20} the authors showed that ADMM is closely 
related to Douglas--Rachford splitting and Peaceman--Rachford splitting, and
established a unifying global convergence result under the only assumption of Lipschitz differentiability 
of one function. 


\subsection{Optimization problem and the PL-ADMM algorithm}
In this paper, we study the global convergence analysis of a variant of ADMM algorithm 
to solve the linearly constrained {\it nonconvex} and {\it nonsmooth} 
 minimization problem
\begin{eqnarray}\label{OP}
\begin{array}{ccc}
&\displaystyle{\min_{x,y}}&
\mathcal F(x,y):=
f(x)+g(x)+h(y)\\
&{\rm s.t.}& Ax+By+c=0, 
\end{array}
\end{eqnarray}
where $x\in\mathbb R^n$ and $y\in  \mathbb R^m$ are unknown variables,
$A\in\mathbb R^{n\times p}$,
$B\in\mathbb R^{m\times p}$,
$c\in\mathbb R$.
The function $f:\mathbb R^n\to \bar{\mathbb R}$ is proper and lower-semicontinuous
while $g:\mathbb R^n\to \bar{\mathbb R}$ and 
$h: \mathbb R^m\to \bar{\mathbb R}$ are proper smooth functions.
We do not assume any convexity assumption on $f$, $g$, and $h$.
The augmented Lagrangian function $\mathcal L^{\alpha} (x,y,z)$ 
associated with the problem (\ref{OP}) is defined by 
\begin{eqnarray}\label{augL}
&\mathcal L^{\alpha}: \mathbb R^n\times\mathbb R^m\times \mathbb R^p\to\mathbb R\nonumber&\\
&\mathcal L^{\alpha} (x,y,z)=f(x)+g(x)+h(y)+\langle z, Ax+By+c\rangle 
+\frac{\alpha}{2}\big\|Ax+By+c\big\|^2,&
\end{eqnarray}
where $\alpha>0$ and $z\in\mathbb R^p$ is the Lagrange multiplier  associated with the linear constraint
$Ax+By+c=0$.
Let $\{Q_1^k\}_{k\ge 0}\subseteq  \mathbb R^{n\times n}$
and $\{Q_2^k\}_{k\ge 0}\subseteq  \mathbb R^{m\times m}$ be
two sequences of symmetric and positive semidefinite matrices.
Given the initial vector $(x^0,y^0,z^0)$ and for $k=1,2, \dots$ until some stopping criterion satisfied 
the variable metric proximal ADMM algorithm generates 
the sequence $\{(x^k,y^k,z^k)\}_{k\ge 0}$ recursively as follows
\begin{eqnarray}\label{PADMM}
\begin{array}{cll}
x^{k+1}&\in&\displaystyle{\arg\min_{x}}\; \mathcal L^{\alpha} (x,y^k,z^k)+\frac 12 \|x-x^k\|^2_{Q_1^k},\\
y^{k+1}&=&\displaystyle{\arg\min_{y}}\; \mathcal L^{\alpha} (x^{k+1},y,z^k)+\frac 12 \|y-y^k\|^2_{Q_2^k},\\
z^{k+1}&=&z^k+\alpha (Ax^{k+1}+By^{k+1}+c),
\end{array}
\end{eqnarray}
where $\|v\|^2_{Q}=\langle v, Q v\rangle$ for any $v\in\mathbb R^d$ and $Q\in\mathbb R^{d\times d}$,
$\langle \cdot, \cdot\rangle$ denotes the Euclidean inner product,
and $\|\cdot\|=\sqrt{\langle\cdot,\cdot\rangle}$ denotes the $\ell_2$ norm. The algorithm (\ref{PADMM}) can be equivalently written as follows
\begin{eqnarray*}
\begin{array}{lll}
x^{k+1}&\in&\displaystyle{\arg\min_{x\in\mathbb R^n} }
f(x)+g(x)+\langle z^k,Ax\rangle 
+\frac{\alpha}{2}\|Ax+By^k+c\|^2+\frac 12 \|x-x^k\|^2_{Q_1^k}, 
\\
y^{k+1}&=&\displaystyle{\arg\min_{y\in\mathbb R^m} }
h(y)+ \langle z^k,By\rangle 
+\frac{\alpha}{2}\|By+Ax^{k+1}+c\|^2+\frac 12 \|y-y^k\|^2_{Q_2^k}, 
\\
z^{k+1}&=&z^k+\alpha(Ax^{k+1}+By^{k+1}+c).
\end{array}
\end{eqnarray*}
For efficiency, we take advantage of the smooth structure of $g(\cdot)$,  $\frac{\alpha}{2}\|A \cdot +By^k+c\|^2$,
and $h(\cdot)$, and we replace them by their proper linearizations  to obtain
the variable metric Proximal Linearized ADMM (PL-ADMM) algorithm:
\begin{eqnarray}\label{Alg1}
\begin{array}{l}
x^{k+1}\in
\displaystyle{\arg\min_{x\in\mathbb R^n}\; \hat f^k(x)}\\
\hspace{.2in}\hat f^k(x):=
f(x)+
\big\langle \nabla g(x^k)+\alpha A^*\big(Ax^k+By^k+c+{\alpha}^{-1} z^k\big),x-x^k\big\rangle
+\frac 12\|x-x^k\|^2_{Q_1^k}, 
\\
y^{k+1}=\displaystyle{
\arg\min_{y\in\mathbb R^m}}\;\hat h^k(y)
\\
\hspace{.2in}\hat h^k(y):=
\big\langle \nabla h(y^k)+B^*z^k, y-y^k\big\rangle
+\frac{\alpha}{2}\|By+Ax^{k+1}+c\|^2+\frac 12 \|y-y^k\|^2_{Q_2^k},
\\[.1in]
z^{k+1}=z^k+\alpha \beta(Ax^{k+1}+By^{k+1}+c),
\end{array}
\end{eqnarray}
where $\beta\in(0,2)$ is an over-relaxation parameter.
The PL-ADMM algorithm (\ref{Alg1}) is related but different from \cite{IEEENonADMM19, ProxADMMnoncvx19}.
Work \cite{LiuShenGu19} considers PL-ADMM with the proximal terms 
$\frac{L_x}{2}\|x-x^k\|^2$ and $\frac{L_y}{2}\|y-y^k\|^2$, where $L_x>0$ and $L_y>0$ 
are fixed positive constants, and $\beta=1$. {\it Algorithm\;2} in  \cite{ProxADMMnoncvx19} does 
not exploit lineariziation of $\frac{\alpha}{2} \|A\cdot+By^k+c\|^2$ in the $x$ subproblem,
and solves (\ref{OP}) with $g(x)=0$, $A=-I_{n\times n}$ and $c=0$.

Note that wise choices of proximal matrices $\{Q_i^k\}_{k\ge 0}$ for $i=1,2$ can lead us to much easier
computations for $x^{k+1}$ and $y^{k+1}$, consequently might yield a more efficient scheme. 
For instance, by setting $Q_1^k=\frac {1}{t_k} I_n$ where $\{t_k\}_{k\ge 0}$ is a positive sequence 
and $I_n$ is an $n\times n$ identity matrix, the $x$ subproblem in (\ref{Alg1}) becomes 
the following prox-linear problem 
\begin{eqnarray*}\label{prox}
x^{k+1}&:=&\arg\min_{x\in\mathbb R^n}
\Big\{f(x)+\big\langle p^k,x-x^k\big\rangle+\frac{1}{2t^k}\|x-x^k\|^2\Big\}\\
&=& \arg\min_{x\in\mathbb R^n}
\Big\{f(x)+\frac{1}{2t^k}\|x-x^k+t^kp^k\|^2\Big\}
\end{eqnarray*}
where $p^k:=\nabla g(x^k)+\alpha A^*(Ax^k+By^k+c+\alpha^{-1}z^k)$. 
Prox-linear subproblems can be easier 
to compute specially when $f$ is a separable function.
The $y$ subproblem in (\ref{Alg1}) contains the second order term $\frac{\alpha}{2} y^* B^*B y$. 
If $B^*B $ is nearly a diagonal matrix (or nearly an orthogonal matrix),
 one can replace $B^*B$ by a certain symmetric diagonal (orthogonal) matrix $D\approx B^*B$.
This replacement gives rise to $\frac{\alpha}{2} y^* B^*B y=\frac{\alpha}{2} y^* D y^*$,
and then one can choose $Q_2^k= \alpha (D-B^*B)$ for efficiency.

\subsection{Main contribution}
The main contribution of this paper is the establishment of 
theoretical convergence and rate analysis of the PL-ADMM 
algorithm, introduced in  (\ref{Alg1}). We prove 
\begin{itemize}
\item  the PL-ADMM sequence  is 
bounded  (Theorem \ref{thmbdd}).
\item any limit point of the PL-ADMM sequence is a stationary point (Lemma \ref{thm-cluster}).
\item the PL-ADMM sequence is Cauchy, hence it converges (Theorem \ref{conv}).
\item the rate of convergence for the error of regularized augmented Lagrangian
(Theorem \ref{FR}). By Lemma \ref{samelimit}, this provides the rate of convergence for the error of 
objective function.
\item the convergence rate for the sequential error (Theorem \ref{seqrate}).
\end{itemize}

\subsection{Notation }
Throughout this paper, we denote $\mathbb R$ as the real number set
while $\mathbb Z$ as the set of integers. 
The set $\bar{\mathbb R}:=\mathbb R\cup\{+\infty\}$ is the extended real number,
$\mathbb R_+$ is the positive real number set,
and $\mathbb Z_+$ is the set of positive integers.
Given the matrix $X$, ${\rm Im}(X)$ denotes its image.
We denote by $I_n$ the $n\times n$ identity matrix for $n\in\mathbb Z_+$. The minimum, maximum,
and the smallest positive eigenvalues of the matrix $X\in\mathbb R^{n\times n}$ are denoted by 
$\lambda_{\min}^X$, $\lambda_{\max}^X$, $\lambda_+^X$, respectively.
The Euclidean scalar product of $\mathbb R^n$ and its corresponding 
norms are, respectively, denoted by $\langle \cdot, \cdot\rangle$
and $\|\cdot\|=\sqrt{\langle \cdot,\cdot\rangle}$.
 If  $n_1,\dots, n_p\in\mathbb Z_+$ and $p\in\mathbb Z_+$,
then for any $v:=(v_1,\dots,v_p)\in\mathbb R^{n_1}\times\mathbb R^{n_2}\times\dots\times\mathbb R^{n_p}$ and $v':=(v'_1,\dots,v'_p)\in \mathbb R^{n_1}\times\mathbb R^{n_2}\times\dots\times\mathbb R^{n_p}$ the Cartesian product and its norm are defined by
\[
\ll v,v'\gg =\sum_{i=1}^{p}\langle v_i,v_i'\rangle\quad\quad
\frac{1}{\sqrt{p}}  \sum_{i=1}^{p} \| v_i\| \le |||v||| = \sqrt{\sum_{i=1}^{p} \| v_i\|^2}\le  \sum_{i=1}^{p} \| v_i\|.
\]
For the sequence $\{u^k\}_{k\ge 1}$, $\Delta u^{k}:=u^k-u^{k-1}$,
for all $k\ge 1$. 

\subsection{Basic results on nonsmooth analysis}
 Let $\Phi:\mathbb R^d\to\bar{\mathbb R}$ be a proper and lower semicontinuous function. 
 The domain  of $\Phi$, denoted ${\rm dom}\; \Phi$, is defined by 
$
{\rm dom}\; \Phi:=\{x\in\mathbb R^d:\; \Phi(x)<+\infty\}.
$
For any $x\in {\rm dom}\; \Phi$, the {\it Fr\'echet (viscosity) subdifferential} of $\Phi$ at $x$,
denoted $\hat\partial\Phi(x)$, is defined by
\[
\hat\partial\Phi(x)=
\Big\{
s\in\mathbb R^d: \;\;\lim_{y\neq x}\inf_{y\to x} \frac{\Phi(y)-\Phi(x)-\langle s, y-x\rangle}{\|y-x\|}\ge 0
\Big\}.
\]
For $x\notin {\rm dom}\;\Phi$, then $\hat\partial\Phi(x) =\emptyset$.
The {\it limiting (Mordukhovich) subdifferential}, or simply the subdifferential for short,
 of $\Phi$ at $x\in {\rm dom}\;\Phi$, denoted $\partial\Phi(x)$, is defined by
\begin{eqnarray*}
\partial\Phi(x):=\{s\in\mathbb R^d: \exists x^k\to x,\;\Phi(x^k)\to\Phi(x)\;\;
{\rm and} \;\; s^k\in\partial\hat\Phi(x^k)\to s\;\;{\rm as}\; k\to+\infty\}.
\end{eqnarray*}
 For any $x\in\mathbb R^d$, the above definition implies 
$\hat\partial\Phi(x) \subset \partial \Phi(x)$,
where the first set is convex and closed while the second 
one is closed (\cite{RockafellarWets}, Theorem 8.6).
 When $\Phi$ is convex the two sets coincide and
\[
\hat\partial\Phi (x)=\partial\Phi (x) =\{s\in\mathbb R^d: \Phi(y)\ge \Phi(x)
+\langle s, y-x\rangle\; \forall y\in\mathbb R^d\}.
\]
Let $(x^k,s^k)\in {\rm Graph}\; \partial\Phi:=\{(x,s)\in\mathbb R^d\times\mathbb R^d: s\in\partial \Phi(x)\}$
and $\lim_{k\to\infty}(x^k,s^k)=(x^*,s^*)$. Since $s^k\in\partial \Phi(x^k)$
and $\lim_{k\to\infty}\Phi(x^k)=\Phi(x^*)$ we have  $(x^*,s^*)\in {\rm Graph}\; \partial\Phi$.

The well-known Fermat's rule ``$x\in\mathbb R^d$ is a local minimizer of $\Phi$, then $\partial \Phi(x)\ni 0$''
remains unchanged. If $x\in\mathbb R^d$ such that $\partial \Phi(x)\ni 0$ the point $x$ is called 
a critical point. We denote by ${\rm crit}\; \Phi$ the set of {\it critical points} of $\Phi$,
that is 
\[
{\rm crit}\; \Phi=\{x\in\mathbb R^d:0\in\partial\Phi(x)\}.
\]
 Let $\Omega$ be a subset of $\mathbb R^d$  and $x$ be any point in $\mathbb R^d$. The distance from $x$
to $\Omega$, denoted ${\rm dist}(x,\Omega)$, is defined by 
${\rm dist}(x,\Omega)=\inf\{\|x-z\|:\;z\in\Omega\}.$
If $\Omega=\emptyset$, then ${\rm dist}(x,\Omega)=+\infty$ for all $x\in\mathbb R^d$. 
For any real-valued function $\Phi$ on $\mathbb R^d$ we have
\[
{\rm dist}(0,\partial\Phi(x))=\inf\{\|s^*\|: \; s^*\in\partial \Phi(x)\}
\]  
Let $F:\mathbb R^n\times\mathbb R^m\to(-\infty,+\infty]$ be a lower semicontinuous function. The subdifferentiation of $F$ at the point $(\hat x,\hat y)$ is defined by 
$
\partial F(\hat x, \hat y)=\Big(\partial_x F(\hat x,\hat y), \partial_y F(\hat x,\hat y)\Big),
$
where $\partial_x F$ and $\partial_y F$ are respectively the differential of the function 
$F(\cdot,y)$ when $y\in\mathbb R^m$ is fixed, and $F(x,\cdot)$ when $x$ is fixed.

\begin{lemma}
Let $\Phi:\mathbb R^d\to\mathbb R$ be Fr\'echet differentiable such that 
its gradient is Lipschitz continuous with constant $L_{\Phi}>0$. Then the following statements are true
\begin{itemize}
\item [a.] For every $u,v\in\mathbb R^d$ and every $\xi\in[u,v]=\{(1-t)u+tv:\;t\in[0,1]\}$
it holds 
\begin{eqnarray}\label{Lip}
\Phi(v)\le \Phi(u)+\langle \nabla \Phi(\xi),v-u\rangle +\frac{L_{\Phi}}{2}\|v-u\|^2.
\end{eqnarray}
%
\item [b.] If $\Phi$ is bounded from below, then the term
$\Phi(v)-(\delta-\frac{L_{\Phi} \delta^2}{2})\|\nabla \Psi(v)\|^2$
is bounded from below for every $\delta>0$.
\end{itemize}
\end{lemma}
\begin{proof}
a. Let $u, v\in\mathbb R^d$, and $\xi=(1-t)u+tv$ for $t\in [0,1]$. Then we have
\begin{eqnarray*}
\Phi(v) -\Phi(u) &=& \int_0^1 \langle \nabla \Phi((1-r) u+rv), v-u \rangle dr \\
&=&  \int_0^1  \langle \nabla \Phi((1-r) u+rv)-\nabla \Phi(\xi), v-u \rangle dr
+\langle \nabla \Phi(\xi) , v-u\rangle 
\end{eqnarray*}

Since $\Phi$ has Lipschitz continuous gradients with constant $L_{\Phi}>0$, then we have
\begin{eqnarray*}
|\Phi(v) -\Phi(u) - \langle \nabla \Phi(\xi) , v-u\rangle|
&=&\Big| \int_0^1  \langle \nabla \Phi((1-r) u+rv)-\nabla \Phi(\xi), v-u \rangle dr\Big | \\
&\le & \int_0^1  \Big\|  \nabla \Phi((1-r) u+rv)-\nabla \Phi(\xi)\Big\|\cdot \Big\|v-u\Big \|dr \\
&\le & L_{\Phi} \|v-u\|^2 \int_0^1 |r-t| dr \\
&=&  L_{\Phi} 
\|v-u\|^2 \Big( \int_0^t (-r+t) dr + \int_t^1 (r-t) dr \Big) \\
&=&  L_{\Phi} \Big(\frac 12 -t(1-t)\Big) \|v-u\|^2 \\
&\le & \frac {L_{\Phi}}{2}\|v-u\|^2.
\end{eqnarray*}

b. We set $\xi=u$ and $v=u-\delta \nabla\Phi(u)$ in (\ref{Lip}) to get
\begin{eqnarray*}
\inf_{v}\Phi(v)\le \Phi \big(u-\delta \nabla\Phi(u)\big)\le \Phi(u) -(\delta-\frac{L_{\Phi}\delta^2}{2})\|\nabla \Phi(u)\|^2.
\end{eqnarray*}
Since $\Phi$ is bounded below, then the right hand side is bounded from below for any $\delta>0$.
 $\blacksquare$
\end{proof}
%

\subsection{\bf Kurdyka-{\L}ojasiewicz (K{\L}) properties}\label{sub:KL}
Our main theoretical results are based on Kurdyka-{\L}ojasiewicz  (K{\L}) properties,
thus in the following we recall some related definitions and facts. 

Let $\Phi:\mathbb R^d\to \bar{\mathbb R}$  be a proper lower semicontinuous function. 
For $-\infty<\eta_1<\eta_2\le +\infty$, we define 
\[
[\eta_1<\Phi<\eta_2] =\{x\in\mathbb R^d: \; \eta_1<\Phi(x)<\eta_2\}.
\]
Let $\eta\in (0,+\infty]$. We denote by $\Psi_{\eta}$ the set of all continuous and concave functions 
 $\psi:[0,\eta]\to [0,+\infty)$ where  $\psi(0)=0$, $\psi$ is continuously differentiable on $(0,\eta)$,
 and $\psi'(s)>0$ over $(0,\eta)$.

\begin{definition}\label{KLprop}
Let $\Phi:\mathbb R^d\to \bar{\mathbb R}$  be a proper  lower semicontinuous function.
We say that $\Phi$ has the Kurdyka-{\L}ojasiewicz (K{\L}) property at 
$x^*\in {\rm dom} \;\partial \Phi$ if there exists a neighborhood $U$ of $x^*$, 
$\eta\in(0,\infty)$ and a continuous concave function $\psi\in\Psi_{\eta}$ such that
\begin{eqnarray}\label{KLine}
\psi'\big(\Phi(x)-\Phi(x^*)\big) {\rm dist}\big(0,\partial \Phi(x)\big)\ge 1.
\end{eqnarray}
 If $\Phi$ satisfies the  property at each point of $ {\rm dom} \;\partial \Phi$, then $\Phi$ is a 
{K\L} function. 
\end{definition}

\begin{definition}\label{KLset}
Let $\Phi:\mathbb R^d\to \bar{\mathbb R}$  be a proper  lower semicontinuous function that
 takes constant value on $\Omega$ and satisfies the {K\L} property at each point of $\Omega$.
We say $\Phi$ satisfies a K{\L} property on $\Omega$ if there exists
$\epsilon>0$, $\eta>0$, and $\psi\in\Psi_{\eta}$ such that 
for every $x^*\in\Omega$ and every element $x$ belongs to the intersection
$ \{x\in\mathbb R^d: {\rm dist} (x,\Omega)<\epsilon\} \cap [\Phi(x^*) <\Phi(x) <   \Phi(x^*)+\eta],$
 (\ref{KLine}) holds. 
 \end{definition}
 
\begin{definition}\label{KLexponent}
Let the proper lower semicontinuous function $\Phi$ satisfying 
the K{\L} property at $x^*\in {\rm dom}\partial \Phi$,
and the corresponding function $\psi\in\Psi_{\eta}$ can be chosen as $\psi(s)=\bar c s^{1-\theta}$
for some $\bar c>0$ and $\theta\in [0,1)$, i.e., there exist $c>0$ and a neighborhood $U$ of $x^*$
and $\eta\in (0,+\infty]$ such that  
\begin{eqnarray}\label{KLexp}
(\Phi(x)-\Phi(x^*))^{\theta}\le c\; {\rm dist}(0,\partial \Phi(x)),\quad\quad \forall x\in U
\end{eqnarray}
holds. Then we say $\Phi$ has the K{\L} property at $x^*$ with an exponent $\theta$.
This asserts that $(\Phi(x)-\Phi(x^*))^{\theta}/{\rm dist}(0,\partial \Phi(x)) $ remains
bounded around $x^*$.
\end{definition}

This definition encompasses broad classes of functions arising in practical 
optimization problems. For example,
 if $f$ is a proper closed semi-algebraic function, then
$f$ is a K{\L} function with exponent $\theta\in [0,1)$ \cite{DM-KL-13}.
The function $l(Ax)$, where $l$ is strongly convex on any compact set 
and it is twice differentiable, and $A\in\mathbb R^{m\times n}$, is a K{\L} function.
Convex piecewise linear-quadratic function, 
$\|x\|_1$, $\|x\|_0$, $\gamma \sum_{i=1}^k |x_{[i]}|$ where $|x_{[i]}|$
is the $i$th largest (in magnitude) entry in $x$, $k\le n$ and $\gamma \in (0,1]$,
$\delta_{\Delta}(x)$ where $\Delta=\{x\in\mathbb R^n: e\tr x=1, x\ge 0\}$,
and least square problems with SCAD \cite{Fan97}
or MCP \cite{Zhang10} regularized functions are  all K{\L} 
 functions.
The K{\L} property characterizes the local geometry of a function around the 
set of critical points. If $\Phi$ is differentiable then $\partial \Phi(x)=\nabla \Phi(x)$,
the inequality (\ref{KLexp}) becomes 
$
\Phi(x)-\Phi(x^*)\le c \|\nabla \Phi(x)\|^{\frac {1}{\theta}},
$
which generalizes the Polyak-{\L}ojasiewics condition 
$\Phi(x)-\Phi(x^*)\le \mathcal O(\|\nabla \Phi(x)\|^2)$ ($\theta=\frac 12$).
We refer interested readers to 
\cite{AttouchBolte09, PAM-KL-10, DM-KL-13,LojasiewicsSmoothsubana06,
CharacLineq10,PALM14,Clarkesubgradients07,FrankelGuillaumePey15}
for more properties of {K\L} functions and illustrating examples.

\subsection{Organization} The paper is organized as follows.
 In Section \ref{sec2} we drive some fundamental properties of the sequence
 generated by the PL-ADMM algorithm, including subgradient and dual 
 bounds,  limiting continuity, and change of the augmented Lagrangian during each variable 
 update.  In Section \ref{sec4}, we introduce the regularized augmented Lagrangian
 and some sufficient decrease condition which helps us proving that the PL-ADMM sequence is bounded.
In Section \ref{sec4} we prove the convergence and the convergence rates of the PL-ADMM  based on the  (K{\L}) properties. Section \ref{sec5} concludes the paper.

 \section{Augmented Lagrangian based Properties}\label{sec2}
In this section, we establish some important properties for the PL-ADMM algorithm (\ref{Alg1})
based on the augmented Lagrangian functional (\ref{augL}).

\begin{assumption}\label{assumptionA}
We begin by making some assumptions. \\
{A1.} The function $f$ is lower-semicontinuous and coercive;\\
A2. $g$ is coercive and $h$ is bounded from below; \\
{A3.} The matrix $B$ is full rank, ${\rm Im}({A})\subseteq {\rm Im}(B)$ and $b\in {\rm Im}(B)$; \\
{A4.} $g$ and $h$ have respectively $L_g$ and  $L_h$ Lipschitz continuous gradients;\\
{A5.} $\alpha>0$, $\beta\in(0,2)$, ${q_i}^-:=\inf_{k\ge 0} \|Q_i^k\|\ge 0$ and $q_i:=\sup_{k\ge 0} \|Q_i^k\|<+\infty$, $i=1,2$
 \end{assumption}

\begin{lemma}\label{thm-d}
{\bf (Subgradient bound)} Suppose that the Assumption \ref{assumptionA} holds. Let $\{(x^k,y^k,z^k)\}_{k\ge 0}$ be a sequence generated by the PL-ADMM algorithm (\ref{Alg1}). There exists 
a constant $\rho>0$ and $d^{k+1}:=(d^{k+1}_x,d^{k+1}_y,d^{k+1}_z)\in\partial \mathcal L^{\alpha}(x^{k+1},y^{k+1},z^{k+1})$ 
such that 
\begin{eqnarray}\label{nd}
|||d^{k+1}||| &\le&
 \rho \Big(\|\Delta x^{k+1}\|+ \|\Delta y^{k+1}\|+ \|\Delta z^{k+1}\|\Big), 
\end{eqnarray}
where for any sequence $\{u^k\}_{k\ge 0}$, $\Delta u^{k+1}=u^{k+1}-u^{k}$, and
\begin{eqnarray}\label{rho}
\rho:=\max\{
q_1+L_g+\alpha\|A\|^2,
\alpha \|A\| \|B\|+L_h+q_2,
 \| A\|  +\|B\|+\frac{1}{\alpha\beta}\}.
\end{eqnarray}

\end{lemma}

\begin{proof}
Let $k\ge 0$ be fixed. By taking partial differential of $\mathcal L^{\alpha}$ with respect to $x$,
and evaluating the result at the point $(x^{k+1},y^{k+1},z^{k+1})$ yields
\[
\partial_x \mathcal L^{\alpha}(x^{k+1},y^{k+1},z^{k+1})=
\partial f(x^{k+1})+\nabla g(x^{k+1})+
\alpha A^*\big(Ax^{k+1}+By^{k+1} +\alpha^{-1}z^{k+1}+ c\big).
\]
By the optimality condition of $x$ subproblem in (\ref{Alg1}) we have
\[
-\nabla g(x^k)-\alpha A^*\big(Ax^{k}+By^k+\alpha^{-1}z^k+ c\big)
-  Q_1^k \Delta x^{k+1}  \in \partial f(x^{k+1}). 
\]
Therefore, we obtain
\begin{eqnarray}\label{dxinL}
d_x^{k+1}&:=&\nabla g(x^{k+1})-\nabla g(x^k)
+A^*\Delta z^{k+1} +\alpha A^*B\Delta y^{k+1}+\alpha A^*A\Delta x^{k+1} - Q_1^k \Delta x^{k+1}
\nonumber \\
&\in&
\;\partial_x \mathcal L^{\alpha}(x^{k+1},y^{k+1},z^{k+1}). 
 \end{eqnarray}
Taking partial differential of $\mathcal L^{\alpha}$ with respect to $y$
and evaluating the result at $(x^{k+1},y^{k+1},z^{k+1}) $ gives
\[
\nabla_y \mathcal L^{\alpha}(x^{k+1},y^{k+1},z^{k+1}) =
\nabla h(y^{k+1})+\alpha B^*\big(Ax^{k+1}+By^{k+1}+\alpha^{-1}z^{k+1}+c\big).
\]
The optimality criterion of $y$ subproblem in (\ref{Alg1}) is given by
\[
\nabla h(y^k)
+ \alpha B^*(Ax^{k+1}+By^{k+1}+\alpha^{-1}z^k+c) 
+ {Q_2^k} \Delta y^{k+1} =0. 
\]
Thus, we have
\begin{eqnarray}\label{dyinL}
d_y^{k+1}:=\nabla h(y^{k+1})-\nabla h(y^k)
+B^*\Delta z^{k+1}- {Q_2^k} \Delta y^{k+1}
\in \nabla_y \mathcal L^{\alpha}(x^{k+1},y^{k+1},z^{k+1}).
\end{eqnarray}
By the $z$ subproblem in  the algorithm (\ref{Alg1}) 
it is easy to see that 
\begin{eqnarray}\label{dzinL}
d_z^{k+1}:=Ax^{k+1}+By^{k+1}+c=\frac{1}{\alpha\beta}
 \Delta z^{k+1}\in \nabla_z \mathcal L^{\alpha}(x^{k+1},y^{k+1},z^{k+1}).
\end{eqnarray}
Hence, by (\ref{dxinL}), (\ref{dyinL}), , and (\ref{dzinL}), we 
then have 
\[
d^{k+1}:=(d^{k+1}_x,d^{k+1}_y,d^{k+1}_z)\in\partial \mathcal L^{\alpha}(x^{k+1},y^{k+1},z^{k+1}).
\]

From (\ref{dxinL}), by using the  triangle inequality we have
\begin{eqnarray*}
\|d_x^{k+1}\|
&\le &\|\nabla g(x^{k+1})-\nabla g(x^k)\|
+\| A\|\cdot \|\Delta z^{k+1}\| +\alpha \|A\| \cdot\|B\| \|\Delta y^{k+1}\| \\
&&+
\alpha \|A\|^2\|\Delta x^{k+1}\|+  \|Q_1^k\| \|\Delta x^{k+1}\|.
\end{eqnarray*}
Since $\nabla g$ is $L_g$ Lipschitz continuous and $q_1=\sup_{k\ge 0} \|Q_1^k\|<+\infty$ we then have 
\begin{eqnarray} \label{ndx}
\|d_x^{k+1}\|\le 
\alpha \|A\| \|B\| \|\Delta y^{k+1}\|
 +(q_1+L_g+\alpha \|A\|^2)\|\Delta x^{k+1}\|
+ \| A\| \|\Delta z^{k+1}\|.
\end{eqnarray}
From (\ref{dyinL}), by the triangle inequality, 
\begin{eqnarray*}
\|d_y^{k+1}\|&\le& \|\nabla h(y^{k+1})-\nabla h(y^k)\|
+\|B\| \|\Delta z^{k+1}\| + \|{Q_2^k}\| \|\Delta y^{k+1}\|. 
\end{eqnarray*}
Since $\nabla h$ is $L_h$ Lipschitz continuous and $q_2=\sup_{k\ge 0} \|Q_2^k\|<+\infty$ 
we get 
\begin{eqnarray} \label{ndy}
\|d_y^{k+1}\|\le  (L_h+q_2)  \| \Delta y^{k+1}\|+ \|B\| \| \Delta z^{k+1}\|.
\end{eqnarray}
We also have
\begin{eqnarray}\label{ndz}
\|d_z^{k+1}\|=\frac{1}{\alpha\beta} \|\Delta z^{k+1}\|.
\end{eqnarray}
By (\ref{ndx}), (\ref{ndy}), and (\ref{ndz}) we obtain
\begin{eqnarray*}
|||d^{k+1}||| &\le&  \|d_x^{k+1}\|+ \|d_y^{k+1}\|+\|d_z^{k+1}\|
\le \rho \big(\|\Delta x^{k+1}\|+  \|\Delta y^{k+1}\|+  \|\Delta z^{k+1}\|\big),
\end{eqnarray*}
where $\rho$ defined in (\ref{rho}). $\blacksquare$
\end{proof}

\begin{lemma}\label{thm-cluster0}
{\bf (Limiting continuity)} 
Suppose that the Assumption \ref{assumptionA} holds. 
If $(x^*, y^*, z^*)$ is the limit point of a subsequence 
$\{(x^{k_j},y^{k_j},z^{k_j})\}_{j\ge 0}$, then 
$ \mathcal L^{\alpha} (x^*, y^*, z^*)=\lim_{j\to\infty} \mathcal L^{\alpha} (x^{k_j},y^{k_j},z^{k_j}).$
\end{lemma}
\begin{proof}
Let $\{(x^{k_j},y^{k_j},z^{k_j})\}_{j\ge 0}$ be a subsequence of the sequence generated by the 
PL-ADMM algorithm 
such that
 \[
\lim_{j\to\infty} (x^{k_j},y^{k_j},z^{k_j})=(x^*, y^*, z^*).
\]
The function $f$ is lower semicontinuous hence we have
\begin{eqnarray}\label{eq:flsc0}
f(x^*)\le \lim\inf_{j\to\infty} f(x^{k_j}).
\end{eqnarray}
From the $x$-subproblem in (\ref{Alg1}),
we have $\hat f^k(x^{k+1})\le \hat f^k(x)$ for any $x\in\mathbb R^n$.  
Choose $k=k_j$, $\forall j\ge 0$, and let $x=x^*$ to get

\begin{eqnarray*}
&f(x^{k_j+1})+\Big\langle \nabla g(x^{k_j})+\alpha A^*\big(Ax^{k_j}
+By^{k_j}+\alpha^{-1}z^{k_j}+c\big),\Delta x^{k_j+1}\Big\rangle
+\|\Delta x^{k_j+1}\|^2_{Q_1^{k_j}}&
\\
&\le f(x^*)+\Big\langle \nabla  g(x^{k_j})
+\alpha A^*(Ax^{k_j}+By^{k_j}+\alpha^{-1}z^{k_j}+c),x^*-x^{k_j}\Big\rangle
+\|x^*-x^{k_j}\|^2_{Q_1^{k_j}}.&
\end{eqnarray*}
By the continuity of $\nabla g$ and the fact that the distance between two successive 
iterates approaches zero, taking the limit supremum from the both sides leads to
\begin{eqnarray*}
\begin{array}{lll}
\lim\sup_{j\to \infty} f(x^{k_j+1}) \le f(x^*)+\\
 \lim\sup_{j\to \infty} \Big\{\Big\langle \nabla g(x^{k_j})+\alpha A^*(Ax^{k_j}+By^{k_j}+\alpha^{-1}z^{k_j}+c),x^*-x^{k_j}\Big\rangle+\|x^*-x^{k_j}\|^2_{Q_1^{k_j}}\Big\}.
\end{array}
\end{eqnarray*}
We have $x^{k_j}\to x^*$ as $j\to\infty$ thus the latter inequality  reduces to 
\[
\lim\sup_{j\to \infty} f(x^{k_j+1})\le f(x^*).
\]
Thus, in view of (\ref{eq:flsc0}), we then have
$\lim_{j\to \infty} f(x^{k_j})=f(x^*).$

Since the functions $h(\cdot)$ and $g(\cdot)$ are smooth we  further have
$\lim_{j\to\infty} g(x^{k_j})= g(x^*)$ and $\lim_{j\to\infty} h(y^{k_j})= h(y^*)$.
Thus 
\begin{eqnarray*}
\begin{array}{l}
\lim_{j\to\infty} \mathcal L^{\alpha} (x^{k_j},y^{k_j},z^{k_j})\\
 =\lim_{j\to\infty} 
\Big\{ f(x^{k_j})+g(x^{k_j})+h(y^{k_j})
+\langle z^{k_j}, Ax^{k_j}+By^{k_j}+c\rangle 
+\frac{\alpha}{2}\|Ax^{k_j}+By^{k_j}+c\|^2\Big\}\\
=f(x^*)+g(x^*)+h(y^*)
+\langle z^*, Ax^*+By^*+c\rangle +\frac{\alpha}{2}\|Ax^*+By^*+c\|^2
= \mathcal L^{\alpha} (x^*, y^*, z^*).
\end{array}
\end{eqnarray*}
That completes the proof. $\blacksquare$
\end{proof}

\begin{lemma}\label{thm-cluster}
{\bf (Limit point is critical point)} Suppose that the Assumption \ref{assumptionA} holds. 
Any limit point $(x^*,y^*,z^*)$ of the sequence  $\{(x^k,y^k,z^k)\}_{k\ge 0}$ 
generated by the PL-ADMM algorithm (\ref{Alg1}) is a stationary point. That is,
$0\in\mathcal L^{\alpha} (x^*,y^*,z^*)$, or equivalently

\begin{eqnarray*} 
 0&\in&\partial f(x^*)+\nabla  g(x^*)+A^*z^*,\\
0&=&\nabla h(y^*)+ B^*z^*\\
0&=&Ax^*+By^*+c.
\end{eqnarray*}
\end{lemma}

\begin{proof}
Let $\{(x^{k_j},y^{k_j},z^{k_j})\}_{j\ge 0}$ be a subsequence of $\{(x^k,y^k,z^k)\}_{k\ge 0}$ such that 
$(x^*, y^*, z^*)=\lim_{j\to\infty} (x^{k_j},y^{k_j},z^{k_j}).$
This follows that  $\|\Delta x^{k_j}\|\to 0$
$\|\Delta y^{k_j}\|\to 0$, and $\|\Delta z^{k_j}\|\to 0$ as $j\to\infty$.
By Lemma \ref{thm-cluster0}, $ \mathcal L^{\alpha} (x^{k_j},y^{k_j},z^{k_j})\to  \mathcal L^{\alpha}(x^*, y^*, z^*)$,  $j\to +\infty$.
Let $d^{k_j}\in\partial \mathcal L^{\alpha} (x^{k_j},y^{k_j},z^{k_j})$, by Lemma \ref{thm-d} 
we have 
$|||d^{k_j}\|||\le \rho (\|\Delta x^{k_j}\|+\|\Delta y^{k_j}\|+\|\Delta z^{k_j}\|)$, 
where $\rho>0$ is an scalar. Since $|||d^{k_j}\|||\to 0$ as $j\to\infty$,
hence $d^{k_j}\to 0$. By the closeness criterion of the limiting sub-differential 
we then have $0\in \partial \mathcal L^{\alpha}(x^*, y^*, z^*)$,
or equivalently, $(x^*, y^*, z^*)\in {\rm crit} (\mathcal L^{\alpha})$.
 $\blacksquare$

\end{proof}
\begin{lemma}\label{descentx}
{\bf (descent of $\mathcal L^{\alpha}$ during $x$ update)} 
Suppose that the Assumption \ref{assumptionA} holds. 
For the sequence $\{(x^k,y^k,z^k)\}_{k\ge 0}$  generated by the PL-ADMM algorithm (\ref{Alg1}) 
we have
\begin{eqnarray}\label{Lagxdiff}
\mathcal L^{\alpha}(x^k,y^k,z^k)-\mathcal L^{\alpha}(x^{k+1},y^{k},z^{k})
\ge \frac 12 \|\Delta x^{k+1}\|^2_{A^k},
\end{eqnarray}
where 
$A^k= { Q_1^k} -\Big(\alpha \lambda_{\max}^{A^*A} +L_g\Big) {I_n}$.
\end{lemma}

\begin{proof}
Let $k\ge 0$ be fixed.
From the $x$ iterate of (\ref{Alg1}) it is clear that 
$\hat f^k(x^{k+1})\le \hat f^k(x)$ for any $x\in\mathbb R^n.$
Setting $x=x^k$ gives%
\begin{eqnarray} \label{eq1}
&f(x^{k+1})-f(x^k)+\big\langle \nabla g(x^k),\Delta x^{k+1}\big\rangle
+\frac 12 \|\Delta x^{k+1}\|^2_{Q_1^k} \nonumber&\\
&\le 
-\alpha\big\langle  A^*(Ax^k+By^k+\alpha^{-1}z^k+c),\Delta x^{k+1}\big\rangle.&
\end{eqnarray}
We next consider
\begin{eqnarray*}
&\mathcal L^{\alpha}(x^k,y^k,z^k)-\mathcal L^{\alpha}(x^{k+1},y^{k},z^{k})
=
f({x^k})-f(x^{k+1}) +
g(x^k)-g(x^{k+1})&\\
&-\langle z^k,A\Delta x^{k+1})\rangle
+\frac{\alpha}{2}\|Ax^k+By^k+c\|^2
-\frac{\alpha}{2}\|Ax^{k+1}+By^k+c\|^2&\\
&=f({x^k})-f(x^{k+1})
+g(x^k)-g(x^{k+1})
-\langle z^k,A\Delta x^{k+1}\rangle
&
\\
&+\frac{\alpha}{2}\|A(x^{k}-x^{k+1})\|^2
-\alpha\big\langle  A^*(Ax^{k+1}+By^k+c),\Delta x^{k+1}\big\rangle.&
\end{eqnarray*}
By  (\ref{eq1}) we then have 
\begin{eqnarray*}
&\mathcal L^{\alpha}(x^k,y^k,z^k)-\mathcal L^{\alpha}(x^{k+1},y^{k},z^{k}) \ge& \\
& g(x^k)-g(x^{k+1})
+\big\langle \nabla g(x^k),\Delta x^{k+1}\big\rangle
+\frac{\alpha}{2}\|A\Delta x^{k+1}\|^2  +\frac 12 \|\Delta x^{k+1}\|^2_{Q_1^k}.
\end{eqnarray*}
Since $g$ is $L_g$ Lipschitz continues this yields 
\begin{eqnarray*}\label{Lagxdiff}
\mathcal L^{\alpha}(x^k,y^k,z^k)-\mathcal L^{\alpha}(x^{k+1},y^{k},z^{k})
\ge
\frac 12 \|\Delta x^{k+1}\|^2_{Q_1^k}-\frac{\alpha}{2}\|A\Delta x^{k+1}\|^2 -\frac{L_g}{2}\|\Delta x^{k+1}\|^2.
\end{eqnarray*}
That completes the proof. $\blacksquare$

\end{proof}

\begin{lemma}\label{descenty}
{\bf (descent of $\mathcal L^{\alpha}$ during $y$ update)} 
Suppose that the Assumption \ref{assumptionA} holds. 
For the sequence $\{(x^k,y^k,z^k)\}_{k\ge 0}$  generated by the PL-ADMM algorithm (\ref{Alg1}) 
we have
\begin{eqnarray}\label{Lagydiff}
\mathcal L^{\alpha}(x^{k+1},y^k,z^k)-\mathcal L^{\alpha}(x^{k+1},y^{k+1},z^{k})
\ge \|\Delta y^{k+1}\|_{B^k}^2, 
\end{eqnarray}
where 
${B^k}= {Q_2^k} + (\alpha \lambda_{+}^{B^*B} -L_h){I_m}$, and $\lambda_{+}^{B^*B}$ is the smallest positive
eigenvalue of $B^*B$.
\end{lemma}

\begin{proof}
Let $k\ge 1$ be fixed. From the optimality condition of $y$ subproblem in (\ref{Alg1})
we have
\begin{eqnarray*}
\nabla h(y^{k})+\alpha B^*(By^{k+1}+Ax^{k+1}+\alpha^{-1}z^k+c)+Q_2^k\Delta y^{k+1}=0
\end{eqnarray*}
Multiply this equation by $\Delta y^{k+1}$ and rearrange to obtain
\begin{eqnarray}\label{eq:2}
- \alpha\langle B^*(By^{k+1}+Ax^{k+1}+\alpha^{-1}z^k+c),\Delta y^{k+1}\rangle 
=\langle \nabla h(y^{k}),\Delta y^{k+1}\rangle +\|\Delta y^{k+1}\|_{Q_2^k}^2.
\end{eqnarray}
We next consider
\begin{eqnarray*}
&\mathcal L^{\alpha}(x^{k+1},y^k,z^k)-\mathcal L^{\alpha}(x^{k+1},y^{k+1},z^{k})=&\\
& h(y^k)-h(y^{k+1})
+\frac{\alpha}{2}\|By^k+Ax^{k+1}+c\|^2
-\frac{\alpha}{2}\|By^{k+1}+Ax^{k+1}+c\|^2-\langle z^k, B\Delta y^{k+1} \rangle=&\\
& h(y^k)-h(y^{k+1})
+\frac{\alpha}{2}\|B\Delta y^{k+1}\|^2
-\alpha \langle \Delta y^{k+1}, B^*(By^{k+1}+Ax^{k+1}+\alpha^{-1}z^k+c)\rangle =&\\
&h(y^k)-h(y^{k+1})+\frac{\alpha}{2}\|B\Delta y^{k+1}\|^2
+\langle \nabla h(y^{k}),\Delta y^{k+1}\rangle +\|\Delta y^{k+1}\|_{Q_2^k}^2,&
\end{eqnarray*}
where the last equation is obtained by (\ref{eq:2}). Since $\nabla h$ is $L_h$ Lipschitz continuous we then get 
\begin{eqnarray*}
\mathcal L^{\alpha}(x^{k+1},y^k,z^k)-\mathcal L^{\alpha}(x^{k+1},y^{k+1},z^{k})
\ge   \frac 12 \|\Delta y^{k+1}\|^2_{Q_2^k}
+\frac{\alpha}{2}\|B\Delta y^{k+1}\|^2
-\frac{L_h}{2}\|\Delta y^{k+1}\|^2. 
\end{eqnarray*}
This completes the proof. $\blacksquare$
\end{proof}

\begin{lemma}\label{thm1}
Suppose that the Assumption \ref{assumptionA} holds. 
For the sequence $\{(x^k,y^k,z^k)\}_{k\ge 0}$  generated by the PL-ADMM algorithm (\ref{Alg1}) 
we have
\begin{eqnarray}\label{LagMon}
\mathcal L^{\alpha}(x^{k+1},y^{k+1},z^{k+1})
+\Big\|\Delta x^{k+1}\Big\|^2_{ A^k}
+\Big\|\Delta y^{k+1}\Big\|^2_{B^k} \le \mathcal L^{\alpha}(x^{k},y^k,z^k)
+\frac{1}{\alpha\beta}\Big\|\Delta z^{k+1}\Big\|^2,
\end{eqnarray}
where 
$A^k= { Q_1^k} -\big(\alpha \lambda_{\max}^{A^*A} +L_g\big) {I_n}$
and 
${B^k}= {Q_2^k} + (\alpha \lambda_{+}^{B^*B} -L_h){I_m}.$

\end{lemma}
\begin{proof}
Let $k\ge 1$ be fixed. By the $z$ update in (\ref{Alg1}) and Lemma \ref{descentx} and \ref{descenty} we have 
\begin{eqnarray*}
\mathcal L^{\alpha}(x^{k+1},y^{k+1},z^{k+1})
&=&\mathcal L^{\alpha}(x^{k+1},y^{k+1},z^{k}) + \frac{1}{\alpha\beta}\|\Delta z^{k+1}\|^2\\
&\le& \mathcal L^{\alpha}(x^{k+1},y^{k},z^{k}) -
 \|\Delta y^{k+1}\|^2_{B^k} + \frac{1}{\alpha\beta}\|\Delta z^{k+1}\|^2
 \\
&\le& \mathcal L^{\alpha}(x^{k},y^{k},z^{k})-
 \|\Delta x^{k+1}\|^2_{A^k} -\|\Delta y^{k+1}\|^2_{B^k} + \frac{1}{\alpha\beta}\|\Delta z^{k+1}\|^2.
\end{eqnarray*}
Rearrange to obtain (\ref{LagMon}). $\blacksquare$
\end{proof}

\begin{lemma}\label{thm2}
{\bf (Monotonicity of $\mathcal L^{\alpha}$)} Suppose that the Assumption \ref{assumptionA} holds. 
For the sequence $\{(x^k,y^k,z^k)\}_{k\ge 0}$  generated by the PL-ADMM algorithm (\ref{Alg1}) 
the following inequalities hold
\begin{eqnarray*}\label{zdiff-square}
(a)\quad\quad\quad\frac{1}{\alpha\beta}\Big\|\Delta z^{k+1}\Big\|^2
\le \theta_0  \|\Delta y^k\|^2
+ \theta_1\|\Delta y^{k+1}\|^2
+\gamma_0 \Big\|B^*\Delta z^{k}\Big\|^2
-\gamma_0\| B^*\Delta z^{k+1}\Big\|^2
\end{eqnarray*}
\begin{eqnarray*}\label{MonLag2}
\begin{array}{ccc}
&(b)\quad\quad\quad\mathcal L^{\alpha}\big(x^{k+1},y^{k+1},z^{k+1}\big)
+\big\|\Delta x^{k+1}\big\|^2_{{A^k}}
+\big\|\Delta y^{k+1}\big\|^2_{B^k-r\theta_1{I_m}}
+\displaystyle{\frac{r-1}{\alpha\beta}}\big\|\Delta z^{k+1}\big\|^2
&\\[.1in]
&
+r\gamma_0 \big\|B^*\Delta z^{k+1}\big\|^2
\le \mathcal L^{\alpha}\big(x^{k},y^k,z^k\big)
+r\gamma_0 \big\|B^*\Delta z^k\big\|^2
+r\theta_0 \big\|\Delta y^k\big\|^2,& 
\end{array}
\end{eqnarray*}
where $r >1$, 
\begin{eqnarray}\label{param}
\theta_0:=  \frac{2\beta (L_h+q_2)^2}{\alpha \lambda_{+}^{B^*B} \big(1-|1-\beta|\big)^2},\;
\theta_1:= \frac{2\beta q_2^2}{\alpha\lambda_{+}^{B^*B} \big(1-|1-\beta|\big)^2 }\;
\gamma_0:=\frac{ |1-\beta|}{\alpha\beta \lambda_{+}^{B^*B} \big(1-|1-\beta|\big)}.
\end{eqnarray}

\end{lemma}
\begin{proof}
Part (a). Let $k\ge 1$ be fixed. We define  the  vector 
\begin{eqnarray}\label{w}
w^{k+1}:=-{Q_2^k}\Delta y^{k+1}-\nabla h(y^k).
\end{eqnarray}
Then 
\[
\Delta w^{k+1}= Q_2^{k-1}\Delta y^{k}-{Q_2^k}\Delta y^{k+1}+ \nabla h(y^{k-1})-\nabla h(y^k),
\] and 
by the triangle inequality we have
\begin{eqnarray}\label{deltaw}
\|\Delta w^{k+1}\| &\le& \|\nabla h(y^{k})-\nabla h(y^{k-1})\|+\|{Q_2^k}\|\|\Delta y^{k+1}\|+\|Q_2^{k-1}\|\|\Delta y^{k}\|
\end{eqnarray}
By the fact that $\nabla h$ is $L_h$ Lipschitz continuous and $q_2=\sup_{k\ge 0}\|{Q_2^k}\|<+\infty$ we obtain
\begin{eqnarray}\label{deltawnorm}
\|\Delta w^{k+1}\| 
\le  (L_h+q_2) \|\Delta y^k\| + q_2 \|\Delta y^{k+1}\|.
\end{eqnarray}
and hence
 \begin{eqnarray}\label{deltawnorm2}
\|\Delta w^{k+1}\|^2 
\le 2 (L_h+q_2)^2 \|\Delta y^k\|^2 +2 q_2^2 \|\Delta y^{k+1}\|^2.
\end{eqnarray}
Expressing the optimality condition of $y$ subproblem using $w^{k+1}$ gives
\[
w^{k+1}=\alpha B^*(Ax^{k+1}+By^{k+1}+c+\alpha^{-1}z^k)
\]
Combining this with the $z$ iterate in (\ref{Alg1}) yields 
\begin{eqnarray}\label{eq6}
B^*z^{k+1}=\beta w^{k+1}+(1-\beta) B^*z^k. 
\end{eqnarray}
This follows that 
\begin{eqnarray*}
B^*\Delta z^{k+1}=\beta\Delta w^{k+1}+(1-\beta) B^*\Delta z^{k}.
\end{eqnarray*}
Since $\beta\in(0,2)$, we can equivalently write this as follows
\begin{eqnarray*}
B^*\Delta z^{k+1}= \big(1-|1-\beta|\big)
\Big(\frac{\beta\Delta w^{k+1}}{1-|1-\beta|} \Big)
+|1-\beta| \Big({\rm sign}(1-\beta) B^*\Delta z^k\Big),
\end{eqnarray*}
where ${\rm sign} (\lambda)=1$ if $\lambda\ge 0$ and ${\rm sign} (\lambda)=-1$ if $\lambda<0$.
By the convexity of $\|\cdot\|^2$ 
we have
\begin{eqnarray*}
\lambda_{+}^{B^*B} \Big(1-|1-\beta|\Big)\Big\|\Delta z^{k+1}\Big\|^2
&\le& \Big(1-|1-\beta|\Big)  \Big\| B^*\Delta z^{k+1}\Big\|^2\\
&\le& \frac{\beta^2}{1-|1-\beta|}  \Big\|\Delta w^{k+1}\Big\|^2+|1-\beta| \Big\|B^*\Delta z^{k}\Big\|^2-|1-\beta| \Big\| B^*\Delta z^{k+1}\Big\|^2.
\end{eqnarray*}
Multiply the both sides of the latter inequality by 
\[
\frac{1}{\alpha\beta \lambda_{+}^{B^*B} \Big(1-|1-\beta|\Big)      }
\]
and replacing (\ref{deltawnorm2}) for $\|\Delta w^{k+1}\|^2$ to obtain 
\begin{eqnarray*}
\frac{1}{\alpha\beta}\Big\|\Delta z^{k+1}\Big\|^2
&\le& \frac{2\beta (L_h+q_2)^2}{\alpha \lambda_{+}^{B^*B} \Big(1-|1-\beta|\Big)^2 }  \|\Delta y^k\|^2\nonumber\\
&&+\frac{2\beta q_2^2}{\alpha \lambda_{+}^{B^*B} \Big(1-|1-\beta|\Big)^2 } \|\Delta y^{k+1}\|^2 \nonumber  \\
&&+\frac{ |1-\beta|}{\alpha\beta \lambda_{+}^{B^*B} \Big(1-|1-\beta|\Big) } \Big\|B^*\Delta z^{k}\Big\|^2\nonumber\\
&&-\frac{ |1-\beta|}{\alpha\beta \lambda_{+}^{B^*B} \Big(1-|1-\beta|\Big) }\Big\| B^*\Delta z^{k+1}\Big\|^2.\nonumber
\end{eqnarray*}

Part (b). Multiply (a) by $r>1$ and combine it with (\ref{LagMon})  to obtain the result. 
$\blacksquare$
\end{proof}
%

\section{Regularized Augmented Lagrangian}\label{sec4}

The regularized Augmented Lagrangian functional is defined by 
\begin{eqnarray}
&\mathcal R:\mathbb R^n\times \mathbb R^m\times \mathbb R^p\times \mathbb R^m\times \mathbb R^p
\to (-\infty,+\infty] &  \nonumber\\[.1in]
&\mathcal R(x,y,z,y',z')=\mathcal L^{\alpha}(x,y,z)+
r \gamma_0 \Big\|B^*(z-z')\Big\|^2+
r  \theta_0 \Big\|y-y'\Big\|^2,&\label{eq9}
\end{eqnarray}
where $r>1$, and $\theta_0$ and $\gamma_0$ are as in (\ref{param}).
For any $k\ge 1$, we denote
\begin{eqnarray}\label{eq10}
\mathcal R_k:=\mathcal R(x^k,y^k,z^k,y^{k-1},z^{k-1}) 
=\mathcal L^{\alpha}(x^k,y^k,z^k)+
r \gamma_0 \|B^*\Delta z^k\|^2+
 r\theta_0 \|\Delta y^k\|^2. 
\end{eqnarray}
By (\ref{MonLag2}) we then have
\begin{eqnarray}\label{eq:mdr}
\mathcal R_{k+1}
+\|\Delta x^{k+1}\|^2_{A^k}+\|\Delta y^{k+1}\|^2_{D^k}+\frac{r-1}{\alpha\beta}\|\Delta z^{k+1}\|^2
\le \mathcal R_k.
\end{eqnarray}
where $D^k:= B^k-r(\theta_0+\theta_1) {I_m}.$

\begin{assumption}\label{assumptionB}
{\bf (Sufficient decrease condition)}\\
\begin{itemize}
\item [(i)] The parameters $\alpha>0, \beta \in (0,2), r>1$, and the symmetric positive semidefinite matrices $\{Q_1^k\}_{k\ge 0}$ and $\{Q_2^k\}_{k\ge 0}$ are chosen such that there exists  $\sigma_1>0$ and $\sigma_2>0$ such that 
\[
q_1^{-}-\alpha \lambda_{\max}^{A^*A}-L_g\ge \sigma_1
\quad
{\rm and}
\quad
q_2^{-}+\alpha \lambda_{+}^{B^*B} - (L_h+r\theta_0+r\theta_1)\ge \sigma_2.
\]
We then let  
\[
\sigma=\min\{\sigma_1,\sigma_2, \frac {r-1}{\alpha\beta}\}.
\]
By this assumption, for all $k\ge 1$ 
 \begin{eqnarray}\label{eq:mdr11}
\mathcal R_{k+1}
+\sigma \Big(\|\Delta x^{k+1}\|^2+\|\Delta y^{k+1}\|^2+\|\Delta z^{k+1}\|^2\Big)
\le \mathcal R_k\le \mathcal R_1.
 \end{eqnarray}
\item [(ii)] {\bf (Relaxed Version)}
The conditions in the assumption can be relaxed. 
Chosen the parameters $\alpha>0, \beta \in (0,2), r>1$,
suppose that the symmetric positive semidefinite  matrices$\{Q_1^k\}_{k\ge 0}$ and $\{Q_2^k\}_{k\ge 0}$  are chosen such that after a finite number of iterations $k_0\ge 0$ we have
\[
\sigma_k=\inf_{k\ge k_0}\Big\{\|Q_1^k\|-\alpha\lambda_{\max}^{A^*A}-L_g,
\|Q_2^k\|+\alpha \lambda_+^{B^*B} - (L_h+r\theta_0+r\theta_1),   \frac{r-1}{\alpha\beta}  \Big \}>\sigma>0.
\]
With this, we would then have 
 \begin{eqnarray}\label{eq:mdr11-r}
\mathcal R_{k+1}
+\sigma \Big(\|\Delta x^{k+1}\|^2+\|\Delta y^{k+1}\|^2+\|\Delta z^{k+1}\|^2\Big)
\le \mathcal R_k\le \mathcal R_{k_0},\quad \forall k\ge k_0. 
 \end{eqnarray}
 \end{itemize}
\end{assumption}
\begin{remark}
We make a few remarks regarding to Assumption \ref{assumptionB}.
\begin{itemize}
\item [(i)] The symmetric and positive semidefinite matrices $Q_1^k$ and $Q_2^k$ 
can be chosen fixed for all $k\ge 0$. 
In particular, they can be chosen a multiple of an identity matrix, that is, $Q_1=\eta I_n$ and $Q_2=\mu I_m$ with 
$\eta>\alpha\lambda_{\max}^{A^*A}+L_g$ and $\mu\ge 0$. 
%
\item [(ii)] By Assumption 2 (i) we observe that it is necessary to have a nonzero proximal term in the $x$ subproblem, so that we have 
$q_1^->\alpha \lambda_{\max}^{A^*A}+L_g$.
However, if $B^*B$ is easily invertible the proximal term in the $y$ subproblem can be eliminated
by choosing $Q_2^k$ to be a zero matrix for all $k\ge 0$ ($q_2^-=q_2=0$). Then for any $\beta \in (0,2)$ and $r>1$ the sufficient decrease condition in Assumption 2 (i) becomes
\[
\alpha^2 \lambda_+^{B^*B}-\alpha L_h-r\frac{2\beta L_h^2r}{\lambda_+^{B^*B}\rho(\beta)^2}>0,
\]
where $\rho(\beta)=1-|1-\beta|$ holds if 
\begin{eqnarray}\label{alpha:cond}
\alpha>\Big(1+\sqrt{1+8\beta r/\rho(\beta)^2}\Big)\frac{L_h}{2\lambda_+^{B^*B}}.
\end{eqnarray}
\item [(iii)]
If the $y$ subproblem in PL-ADMM contains the second order term $\frac{\alpha}{2}y^*B^*By$
and if $B^*B$ is nearly diagonal (or orthogonal), we can replace $B^*B$ by a certain symmetric diagonal (orthogonal) matrix $D$ that approximates $B^*B$. By choosing $Q_2^k=\alpha (D-B^*B)$, the second order term 
  $\frac{\alpha}{2}y^*B^*By$ is replaced by  $\frac{\alpha}{2}y^*Dy$. Note that by Assumption A5 
  we have $q_2^-=\alpha\|D-B^*B\|$, and by (\ref{param}) we have $\theta_0\propto \frac{1}{\alpha}$ and $\theta_1\propto \frac{1}{\alpha}$. Hence, for a fixed $\beta\in(0,2)$ and $r>1$, if $\alpha$ chosen large enough, the relation $q_2^-+\alpha\lambda_+^{B^*B}-(L_h+r\theta_0+r\theta_1)>0$ 
  and hence Assumption 2 (i) holds. 
  \item [(iv)] (Application example)
  Consider the problem 
 \[
 \min_{x\in\mathbb R^n}\;\; 
 r(x)+\mathcal \ell(\mathcal A x-b)
 \] 
 where $\mathcal A:\mathbb R^n\to\mathbb R^m$ and $b\in\mathbb R^m$. 
The first term is the regularization term and the second term is the fitting term.
 Regularizers can be some sparsity-including functions,
 $\ell_q$ quasi-norms for $0<q\le 1$, indicator functions, or any lower-semicontinuous 
 (not necessarily convex) functions satisfying {K\L} conditions. 
The fitting term can be least squares, logistic functions, and other smooth functions.  In order to solve this problem, we formulate it as    
 \[
 \min_{x,y\in\mathbb R^n}\;\; 
 r(x)+\mathcal \ell(\mathcal A y-b)\quad s.t. \;\;\; x-y=0
 \] 
 Comparing to (\ref{OP}), $h(y)=\ell(\mathcal A y-b)$, $g(x)=0$, $A=I$, and $B=-I$.
 Note that  A1-A4 hold. We set $Q_1^k=\eta I$ where $\eta>0$.
 Since $B$ is a multiple of identity, we set $Q_2^k=0$.
 The PL-ADMM algorithm to solve this problem is given as follows
\begin{eqnarray*}
 x^{k+1}&=&\arg\min_{x}\;\; r(x)+ \frac{\eta}{2} \|\alpha\eta^{-1}(x^k-y^k+\alpha^{-1}z^k)+x-x^k\|^2,\\
 y^{k+1}&=&  x^{k+1}-\frac{1}{\alpha}\big(
 \mathcal A^*\nabla \ell(\mathcal Ay^k-b)+B^*z^k \big),\\
 z^{k+1}&=&z^k+\alpha\beta(x^{k+1}-y^{k+1}).
 \end{eqnarray*}
Note that $B^*B=I$, $\lambda_+^{B^*B}=1$. 
If $\ell(.)=\|\;.\;\|^2$, then $L_h=\|\mathcal A\|^2$. To guarantee Assumption 2 (i), 
we need to have $\eta>\alpha$ and (\ref{alpha:cond}). Thus we may choose 
\[
\alpha = \|\mathcal A\|^2 \big(1+\sqrt{1+8\beta r/\rho(\beta)^2}\big),
 \quad \eta =\frac 32 \|\mathcal A\|^2 \big(1+\sqrt{1+8\beta r/\rho(\beta)^2}\big)
\]
 for any $r>1$ and $\beta\in (0,2)$.  
\end{itemize}
\end{remark}
\begin{theorem}\label{thmbdd}
{\bf (bounded sequence)} We assume that Assumption~\ref{assumptionA}, and \ref{assumptionB} (ii) hold.
Then sequence $\{(x^k,y^k,z^k)\}_{k\ge 0}$ generated by the PL-ADMM algorithm (\ref{Alg1}) is bounded.
\end{theorem}
\begin{proof}
Let $\{(x^k,y^k,z^k)\}_{k\ge 0}$ be a generated by the PL-ADMM algorithm. By 
(\ref{eq:mdr11-r}) there exists a $k_0\ge 0$ such that 
$\mathcal R_{k+1}\le \mathcal R_{k_0}$ for all $k\ge k_0$. Hence 
\begin{eqnarray}\label{ineq10}
\begin{array}{c}
f(x^{k+1})+g(x^{k+1})+h(y^{k+1})
+\frac{\alpha}{2}\Big\|Ax^{k+1}+By^{k+1}
+\alpha^{-1}z^{k+1}+c\Big\|^2-\frac{1}{2\alpha}\Big\|z^{k+1}\Big\|^2
\\[.1in]
+\sigma \|\Delta x^{k+1}\|^2+
(\sigma+r\theta_0) \|\Delta y^{k+1}\|^2+
\sigma \|\Delta z^{k+1}\|^2
+r\gamma_0\|B^*\Delta z^{k+1}\|^2
\le R_{k_0}.
 \end{array}
\end{eqnarray}
We will next find a lower bound for $-\frac{1}{2\alpha}\|z^{k+1}\|^2$. 
Given $w^{k+1}$ as in (\ref{w}), we rearrange (\ref{eq6}) to obtain
\[
\beta B^*z^{k+1}=\beta w^{k+1}+(1-\beta) B^*(z^k-z^{k+1}).
\]
Since $\beta\in(0,2)$ we can rewrite this equation as follows
\begin{eqnarray*}
\beta B^*z^{k+1}= (1-|1-\beta|) \Big(\frac{\beta w^{k+1}}{1-|1-\beta|}\Big) 
+(|1-\beta|)\Big({\rm sign}(1-\beta) B^*(z^k-z^{k+1})\Big).
\end{eqnarray*}
By the convexity of $\|\cdot\|^2$  we then obtain
\begin{eqnarray}\label{eq8}
\lambda_+^{B^*B}\beta^2 \Big\|z^{k+1}\Big\|^2
\le \frac{\beta^2}{1-|1-\beta|}  \Big\|w^{k+1}\Big\|^2
+|1-\beta| \Big\|B^*\Delta z^{k+1}\Big\|^2
\end{eqnarray}
Use the inequality $\|w^{k+1}\|^2\le 2(q_2+L_h)^2\|\Delta y^{k+1}\|^2+2\|\nabla h(y^{k+1})\|^2$
in (\ref{eq8}), then divide the both sides of the resulting inequality by $-2\alpha\beta^2\lambda_{+}^{B^*B}$ to get 
\begin{eqnarray*}
-\frac{1}{2\alpha}\Big\|z^{k+1}\Big\|^2 
\ge -\vartheta \|\nabla h(y^{k+1})\|^2
-  \theta_2 \|\Delta y^{k+1}\|^2
-\gamma_1 \Big\|B^*\Delta z^{k+1}\Big\|^2,
\end{eqnarray*}
where
\[
\vartheta:= \frac{1}{\alpha(1-|1-\beta|)\lambda_{+}^{B^*B}},\quad 
\theta_2:=\frac{(q_2+L_h)^2}{\alpha(1-|1-\beta|)\lambda_{+}^{B^*B}},\quad
\gamma_1:=\frac{|1-\beta|}{2\alpha\beta^2 \lambda_{+}^{B^*B}}.
\]
Using the latter inequality, (\ref{ineq10}) leads to 
\begin{eqnarray}\label{thmbdd-eq1}
\begin{array}{l}
f(x^{k+1})+g(x^{k+1})+\frac{\alpha}{2}\Big\|Ax^{k+1}+By^{k+1}+\alpha^{-1}z^{k+1}+c\Big\|^2
+(r\theta_0+\sigma-\theta_2)\Big\|\Delta y^{k+1}\Big\|^2\\[.1in]
+\sigma \Big\|\Delta x^{k+1}\Big\|^2
+ (r\gamma_0-\gamma_1) \Big\|B^*\Delta z^{k+1}\Big\|^2
+\sigma\|\Delta z^{k+1}\|^2
\le \mathcal R_{k_0}
-\inf_y\Big\{ h(y)-\vartheta\Big \|\nabla h(y)\Big\|^2 \Big\}.
\end{array}
\end{eqnarray}
By the Assumption \ref{assumptionA}, $\nabla h$ is $L_h$ Lipschitz continuous, then for any $k\ge k_0\ge 0$ and $y\in\mathbb R^m$ it holds
$h(y)\le h(y^k)+\langle \nabla h(y^k),y-y^k\rangle +\frac{L_h}{2}\|y-y^k\|^2.$
If $\delta>0$ be an scalar, setting $y=y^k-\delta \nabla h(y^k)$ yields 
\begin{eqnarray*}
h\Big(y^k-\delta \nabla h(y^k)\Big)\le h(y^k)- \Big(\delta-\frac{L_h\delta^2}{2}\Big)\|\nabla h(y^k)\|^2. 
\end{eqnarray*}
Since $h$ is bounded from below, then we have 
\begin{eqnarray}\label{hbd}
-\infty<\inf\{h(y)- \Big(\delta-\frac{L_h\delta^2}{2}\Big)\|\nabla h(y)\|^2: \; y\in\mathbb R^m\}
\end{eqnarray}
We choose $\delta>0$ such that 
$\vartheta=\delta-\frac{L_h\delta^2}{2}$.  Then (\ref{hbd}) follows that the right hand side of (\ref{thmbdd-eq1}) is finite.
It is easy to verify for any $r>1$ and any $\beta\in(0,2)$, we also have  $r\gamma_0-\gamma_1>0$ and $r\theta_0+\sigma-\theta_1>0$.
 Hence 
\begin{eqnarray}\label{finite}
f(x^{k+1})+g(x^{k+1})+\Big\|Ax^{k+1}+By^{k+1}+\alpha^{-1}z^{k+1}+c\Big\|^2
+\Big\|\Delta y^{k+1}\Big\|^2 +\|\Delta z^k\|^2<+\infty.
\end{eqnarray}
Since $f$ and $g$ are coercive, then  the sequence 
$\{x^k\}_{k\ge k_0}$ and consequently $\{Ax^k\}_{k\ge k_0}$ is  bounded. 
%
By the $z$ iterate of  (\ref{Alg1}) we have 
\[
By^{k+1} = \frac{1}{\alpha\beta} \Delta z^{k+1} -Ax^{k+1}-c.
\]
By (\ref{finite}), $\{\Delta z^k\}_{k\ge 0}$ is bounded and since $B^*B$ is invertible then $\{y^k\}_{k\ge k_0}$ is bounded.
Finally, since $\{Ax^k\}_{k\ge k_0}$ and $\{By^k\}_{k\ge k_0}$ are bounded
and also $\{Ax^{k}+By^{k}+\frac{1}{\alpha}z^{k}\}_{k\ge k_0}$ is bounded, thus 
$\{z^k\}_{k\ge k_0}$ is bounded.

We showed that the sequences $\{x^k\}_{k\ge k_0}$,$\{y^k\}_{k\ge k_0}$, and $\{z^k\}_{k\ge k_0}$
are bounded. Hence, there exists $M_x>0$, $M_y>0$, $M_z>0$ positive scalars such that 
\begin{eqnarray}\label{bound}
\|x^k\|\le M_x,
\quad
\|y^k\|\le M_y,
\quad
\|z^k\|\le M_z,
\quad 
\forall k\ge k_0.
\end{eqnarray}
We denote by
\begin{eqnarray*}
\hat M_x=\max\{\|x^k\|:\; k=0,1, \cdots, k_0-1\},\\
\hat M_y=\max\{\|y^k\|:\; k=0,1, \cdots, k_0-1\},\\
\hat M_z=\max\{\|z^k\|:\; k=0,1, \cdots, k_0-1\}.\\
\end{eqnarray*}
Thus we have
\[
\|x^k\|\le \max\{M_x,\hat M_x\},
\quad
\|y^k\|\le \max\{M_y,\hat M_y\},
\quad
\|z^k\|\le \max\{M_z,\hat M_z\},
\quad 
\forall k\ge 0.
\]
This concludes the proof.  $\blacksquare$
\end{proof}
\begin{remark}
Theorem \ref{thmbdd} was established for Assumption \ref{assumptionB}(ii). 
With Assumption \ref{assumptionB}(i) then (\ref{bound}) holds for $k\ge 0$.

\end{remark}


\begin{lemma}\label{lmma-Rk}
{\bf(Convergence of $\mathcal R_k$)} 
Suppose that the Assumption \ref{assumptionA} and \ref{assumptionB} (i) hold. 
If $\{(x^k,y^k,z^k)\}_{k\ge 0}$ is a sequence generated by the PL-ADMM algorithm (\ref{Alg1})
then the sequence $\{\mathcal R_{k}\}_{k\ge 1}$ is bounded from below and converges.
\end{lemma}
\begin{proof}
Let $k\ge 0$ be fixed. By (\ref{eq10}), we only need to show that 
\[
\mathcal L^{\alpha}(x^k,y^k,z^k)=f(x^k)+g(x^k)+h(y^k)
+\langle z^k, Ax^k+By^k+c\rangle +\frac{\alpha}{2}\|Ax^k+By^k+c\|^2
\]
is bounded from below. Since $\{(x^k,y^k,z^k)\}_{k\ge 0}$ is a bounded sequence by Theorem \ref{thmbdd},
clearly $\langle z^k, Ax^k+By^k+c\rangle$ and $\|Ax^k+By^k+c\|^2$ are bounded for $k\ge 0$.
By Assumption \ref{assumptionA}, $h$ is bounded from below, and 
 since $f$ and $g$ are coercive and $\{x^k\}_{k\ge 0}$ is bounded, then $\{f(x^k)\}_{k\ge 0}$ and $\{g(x^k)\}_{k\ge 0}$ are
 bounded. 
Therefore $\{\mathcal L^{\alpha}(x^k,y^k,z^k)\}_{k\ge 0}$ is bounded from below, and
consequently 
\[
-\infty <\inf\{\mathcal R_{k}:\; k\ge 0\}.
\]
By Assumption \ref{assumptionB} (i), $\{\mathcal R_k\}_{k\ge 1}$ is monotonically decreasing
for all $k\ge 0$. This together with the fact that $\mathcal R_k$ is bounded from below, we conclude that 
$\{\mathcal R_k\}_{k\ge 1}$ is convergent.  $\blacksquare$

\end{proof}

\begin{lemma}\label{consterms}
Suppose that the Assumption \ref{assumptionA} and \ref{assumptionB} (ii) hold. 
If $\{(x^k,y^k,z^k)\}_{k\ge 0}$ is a sequence generated by the PL-ADMM algorithm (\ref{Alg1}),
which assumed to be bounded, we have 
\[
\lim_{k\to \infty} \|\Delta x^{k+1}\|=0,\quad
\lim_{k\to \infty} \|\Delta y^{k+1}\|=0,\quad
\lim_{k\to \infty} \|\Delta z^{k+1}\|=0.
\]
\end{lemma}

\begin{proof}
By summing up  (\ref{eq:mdr11-r}) from 
$k=k_0$ to some $K\ge k_0$ we have 
\[
\sum_{k=k_0}^{K}
\|\Delta x^{k+1}\|^2+\|\Delta y^{k+1}\|^2
+\|\Delta z^{k+1}\|^2\le \frac{1}{\sigma}(\mathcal R_{k_0} - \inf_{k\ge 0} \mathcal R_k)
<+\infty.
\]
We let $K$ approach to infinity, and since $\{(x^k,y^k,z^k)\}_{k\ge 0}$ is bounded we have  
\begin{eqnarray*}
\sum_{k\ge 0}
\|\Delta x^{k+1}\|^2+\|\Delta y^{k+1}\|^2+\|\Delta z^{k+1}\|^2<+\infty.
\end{eqnarray*}
This follows that 
$\|\Delta x^{k+1}\|^2+\|\Delta y^{k+1}\|^2+\|\Delta z^{k+1}\|^2\to 0,$ as $k\to\infty$, and
thus
\[\lim_{k\to \infty} \|\Delta x^{k+1}\|+\|\Delta y^{k+1}\|+\|\Delta z^{k+1}\|\to 0.
\]
This finishes the proof. $\blacksquare$

\end{proof}

\begin{lemma}\label{limitps}
{\bf (properties of limit point set)} Let the Assumptions \ref{assumptionA} and \ref{assumptionB} hold. 
For a bounded sequence $\{(x^k,y^k,z^k)\}_{k\ge 0}$ generated by the PL-ADMM algorithm (\ref{Alg1})
the following are true
\begin{itemize}
\item [(i)] The limit point set of the sequence $\{(x^k,y^k,z^k)\}_{k\ge 0}$, denoted 
$\omega \Big((x^k,y^k,z^k)\}_{k\ge 0}\Big)$, is nonempty, connected and compact.
\item [(ii)] $\displaystyle{\lim_{k\to\infty} {\rm dist} \Big [(x^k,y^k,z^k), \omega \Big((x^k,y^k,z^k)\}_{k\ge 0}\Big)\Big] =0}$.
\item [(iii)] $\omega \Big((x^k,y^k,z^k)\}_{k\ge 0}\Big) \subseteq {\rm crit}\; \mathcal L^{\alpha} $.
\end{itemize}
\end{lemma}
\begin{proof}
These results follow by Lemma \ref{consterms}. 
We omit the proof.
\end{proof}
\begin{lemma}\label{samelimit}
Suppose that the Assumptions \ref{assumptionA}  holds.
If $(x^*,y^*,z^*)$ is a limit point of a subsequence $\{(x^{k_j},y^{k_j},z^{k_j})\}_{j\ge 0}$, then
\[
\mathcal R(x^*,y^*,z^*, y^*,z^*)
=
\mathcal L^{\alpha} (x^*,y^*,z^*) 
=
f(x^*)+g(x^*)+h(y^*).
\]

\end{lemma}
\begin{proof}
Let $\{(x^{k_j},y^{k_j},z^{k_j})\}_{j\ge 0}$ be a subsequence such that $(x^{k_j},y^{k_j},z^{k_j})\to (x^*,y^*,z^*)$
as $j\to\infty$. Hence $\|\Delta y^{k_j}\|\to 0$ and $\|B^*\Delta z^{k_j}\|\le \|B\| \|\Delta z^{k_j}\| \to 0$ as $k\to\infty$
hence
\begin{eqnarray*}
\lim_{j\to\infty} \mathcal R_{k_j} &=&
\lim_{k\to\infty} \mathcal L^{\alpha} (x^{k_j},y^{k_j},z^{k_j}) \\
&=&\lim_{k\to\infty} f(x^{k_j})+g(x^{k_j})+h(y^{k_j})
+\langle z^{k_j}, Ax^{k_j}+By^{k_j}+c\rangle +\frac{\alpha}{2}\|Ax^{k_j}+By^{k_j}+c\|^2.
\end{eqnarray*}
By the $z$ iterate of the algorithm  (\ref{Alg1}), then $\|\Delta z^{k_j+1}\|\to 0$ hence
\[
\|Ax^{k_j}+By^{k_j}+c\|\to 0,\quad {\rm as} \; j\to\infty.
\]
Since $\{z^{k_j}\}_{j\ge 0}$ is a bounded sequence we also have 
\[
\langle z^{k_j}, Ax^{k_j}+By^{k_j}+c\rangle \to 0, \quad k\to\infty
\]
As $g$ and $h$ are smooth and $f$ is lower semicontinuous, then
\begin{eqnarray*}
\lim_{j\to\infty} \mathcal R(x^{k_j},y^{k_j},z^{k_j},y^{k_j},z^{k_j})
&=&\lim_{j\to\infty} \mathcal L^{\alpha}  (x^{k_j},y^{k_j},z^{k_j})\\
&=&\lim_{j\to\infty} f(x^{k_j})+g(x^{k_j})+h(y^{k_j})\\
&=&f(x^*)+g(x^*)+h(y^*).
\end{eqnarray*}
This concludes the proof. $\blacksquare$

\end{proof}

\section{Convergence and Convergence Rates  }\label{sec5}
In this section, we establish the main theoretical results including convergence and convergence rates for both the sequence generated by the PL-ADMM and the regularized augmented Lagrangian. We begin with some important lemmas.

\begin{lemma}\label{lma-D}
Suppose that the Assumption \ref{assumptionA} holds. Let $\{(x^k,y^k,z^k)\}_{k\ge 0}$ be a sequence generated by the PL-ADMM algorithm (\ref{Alg1}). Define
\begin{eqnarray*}
&s_x^{k}:=d^{k}_x,\quad
s_y^{k}:=d^{k}_y+2r\theta_0 \Delta y^{k},\quad
s_z^{k}:=d^{k}_z+2r\gamma_0BB^*\Delta z^{k},&\\
&s_{y'}^{k}:=-2r\theta_0 \Delta y^{k},\quad
s_{z'}^{k}:=-2\gamma_0BB^*\Delta z^{k}&
\end{eqnarray*}
where $(d_x^k,d_y^k,d_z^k)\in\partial \mathcal L^{\alpha}(x^k,y^k,z^k)$.
Then 
\[
s^{k}:=(s^{k}_x, s^{k}_y, s^{k}_z, s^{k}_{y'},s^{k}_{z'}, )\in
\partial \mathcal R(x^{k},y^{k},z^{k},y^{k-1},z^{k-1})
\] for $k\ge 1$, and it holds
\begin{eqnarray}\label{nD}
|||s^{k}|||\le \tilde\rho \Big(\|\Delta x^{k}\|+\|\Delta y^{k}\|+\|\Delta z^{k}\|\Big),
\end{eqnarray}
where 
\begin{eqnarray}\label{trho}
\tilde \rho=\sqrt{3}\rho+4r\max\{\theta_0, \gamma_0\|B\|^2\},
\end{eqnarray}
$r>1$, $\rho$ is given in (\ref{rho}), 
$\theta_0$ and $\gamma_0$ are as (\ref{param}).
\end{lemma}

\begin{proof}
Let $k\ge 1$ fixed, and $(d_x^{k}, d_y^{k}, d_z^{k})\in \partial \mathcal L^{\alpha}(x^k,y^k,z^k)$. 
By taking partial derivatives of $\mathcal R_k$ with respect to $x,y,z,y',z'$ we obtain
\begin{eqnarray*}
s_x^{k}&:=& \partial_x \mathcal R(x^{k},y^{k},z^{k},y^{k-1},z^{k-1})
=\partial_x \mathcal L^{\alpha}(x^{k},y^{k},z^{k})=d_x^{k},\\
s_y^{k}&:=&\nabla_y \mathcal R (x^{k},y^{k},z^{k},y^{k-1},z^{k-1})
=\nabla_y \mathcal L^{\alpha}(x^{k},y^{k},z^{k})+2r\theta_0 \Delta y^k=
d_y^{k}+2r\theta_0 \Delta y^k,\\
s_z^{k}&:=&\nabla_z \mathcal R(x^{k},y^{k},z^{k},y^{k-1},z^{k-1})
=\nabla_z \mathcal L^{\alpha}(x^{k},y^{k},z^{k})
+2r\gamma_0 BB^*\Delta z^{k}=d_z^{k}+2r\gamma_0 BB^*\Delta z^{k},\\
s_{y'}^{k}&:=&\nabla_{y'} \mathcal R(x^{k},y^{k},z^{k},y^{k-1},z^{k-1})
=-2r\theta_0\Delta y^k,\\
d_{z'}^{k}&:=&\nabla_{z'} \mathcal R(x^{k},y^{k},z^{k},y^{k-1},z^{k-1})
=-2r\gamma_0BB^* \Delta z^k.
\end{eqnarray*}
By the triangle inequality we obtain
\begin{eqnarray*}
&\|s_x^{k}\|=\|d^{k}_x\|,\quad 
\|s_y^{k}\|\le \|d^{k}_y\|+2r\theta_0\|\Delta y^k\|,&\\
&\|s_z^{k}\|\le\|d^{k}_z\|+2r\gamma_0\|B\|^2 \|\Delta z^k\|,\quad
\|s_{y'}^{k}\|= 2r\theta_0\|\Delta y^k\|,\quad
\|s_z^{k}\|=2r\gamma_0 \|B\|^2 \|\Delta z^k\|.&
\end{eqnarray*}
By Lemma \ref{thm-d}, this follows that 
\begin{eqnarray*}
|||s^{k}|||&\le&  \|s_x^{k}\|+\|s_y^{k}\|+\|s_z^{k}\|+ \|s_{y'}^{k}\|+\|s_z^{k}\|\\
&\le& \|d^{k}_x\|+  \|d^{k}_y\|+\|d^{k}_z\|+
4r\theta_0\|\Delta y^{k}\|+
4r\gamma_0 \|B\|^2 \|\Delta z^k\|
\\
&\le& \sqrt{3} |||d^k|||+4r\theta_0\|\Delta y^{k}\|+
4r\gamma_0 \|B\|^2 \|\Delta z^k\|\\
&\le& \sqrt{3}\rho \|\Delta x^{k}\| 
+( \sqrt{3}\rho+4r\theta_0)\|\Delta y^{k}\|
+( \sqrt{3}\rho+4r\gamma_0\|B\|^2)\|\Delta z^{k}\|.
\end{eqnarray*}
This concludes the proof. $\blacksquare$
\end{proof}

\begin{lemma}\label{lma-cluster2}
Let the Assumption \ref{assumptionA} and \ref{assumptionB} hold.
If $\{(x^k,y^k,z^k)\}_{k\ge 0}$ is a sequence generated by the PL-ADMM algorithm (\ref{Alg1}),
then the following statements hold.
\begin{itemize}
\item [(i)] The set 
$\omega \Big(\{(x^k,y^k,z^k,y^{k-1},z^{k-1})\}_{k\ge 1}\Big)$
is nonempty, connected, and compact. 
\item [(ii)] $\Omega \subseteq \{(x, y, z,  y, z)
\in\mathbb R^n\times\mathbb R^m\times\mathbb R^p\times\mathbb R^m\times\mathbb R^p:
\;(x, y, z)\in {\rm crit} (\mathcal L^{\alpha})\Big\}.$
\item [(iii)] $\lim_{k\to\infty} {\rm dist} \Big[ (x^k, y^k, z^k, y^{k-1}, x^{k-1}), \Omega \Big]=0$;
\item [(iv)] The sequences $\{\mathcal R_k\}_{k\ge 0}$, $\{\mathcal L^{\alpha}(x^k,y^k,z^k)\}_{k\ge 0}$, and 
$\{\mathcal F(x^k,y^k)\}_{k\ge 0}$ approach to the same limit and if $(x^*, y^*, z^*,  y^*, z^*)\in\Omega$,
then 
\[
\mathcal R (x^*, y^*, z^*,  y^*, z^*)=
 \mathcal L^{\alpha}(x^*, y^*, z^*)
= \mathcal F(x^*,y^*).
\]

\end{itemize}
\end{lemma}
\begin{proof}
These results follow immediately from Lemma \ref{thm-cluster}, Theorem \ref{lmma-Rk}, 
and Lemma \ref{lma-D}. $\blacksquare$
\end{proof}

\begin{theorem}\label{conv}
{\bf (Convergence)}
Suppose that the Assumptions \ref{assumptionA} and \ref{assumptionB} (ii) hold,
$\{(x^k,y^k,z^k)\}_{k\ge 0}$ is a sequence generated by the
PL-ADMM algorithm (\ref{Alg1}) which is assumed to be bounded, and 
$\mathcal R$ satisfies the
{K\L} property on 
$
\Omega:=\omega \Big(\{(x^k,y^k,z^k,y^{k-1},z^{k-1})\}_{k\ge 1}\Big).
$
That is, for every 
$v^*:=(x^*, y^*, z^*,  y^*, z^*) \in \Omega$
there exists $\epsilon>0$, $\eta\in [0,+\infty)$, and desingularizing function
$\psi\in\Psi_{\eta}$ such that for every  
$v=(x,y,z,y',z')\in\mathcal S$, where 
\begin{eqnarray}\label{eq:intersec}
\mathcal S:=\Big\{v \in \mathbb R^n\times\mathbb R^m\times\mathbb R^p\times\mathbb R^m\times\mathbb R^p:
\;{\rm dist}(v,\Omega)<\epsilon \;{\rm and} \; \mathcal R(v^*) <\mathcal R(v)< \mathcal R(v^*)+\eta\Big\},
\end{eqnarray}
it holds 
\[
\psi'\Big(\mathcal R(v)-\mathcal R(v^*)\Big) {\rm dist}\Big(0,\partial \mathcal R(v)\Big)\ge 1.
\]
Then $\{u^k\}_{k\ge 0}:=\{(x^k,y^k,z^k)\}_{k\ge 0}$ satisfies the finite length property 
\[
\sum_{k=0}^{\infty} 
\| \Delta x^{k}\|
+ \| \Delta y^{k}\|
+ \| \Delta z^{k}\|<+\infty,
\]
and consequently converges to a stationary point of (\ref{OP}).

%
\end{theorem}
\begin{proof}
By Lemma \ref{lmma-Rk}, there exists a $k_0\ge 0 $ such that
the sequence $\{\mathcal R_k\}_{k\ge k_0}$ is monotonically 
decreasing and converges, let $\mathcal R_{\infty}:=\lim_{k\to\infty}\mathcal R_k$. This follows that 
the error sequence $\mathcal E_k:=\mathcal R_k-\mathcal R_{\infty}$,
is non-negative, monotonically decreasing for all $k\ge k_0$, and converges to $0$.
Let us consider two cases:

{\it Case 1.}  There is $k_1\ge k_0$ such that $\mathcal E_{k_1}=0$.
Hence $\mathcal E_k=0$ for all $k\ge k_1$
and by (\ref{eq:mdr11-r}) we have
\[
\|\Delta x^{k+1}\|^2+\|\Delta y^{k+1}\|^2+\|\Delta z^{k+1}\|^2 \le 
 \frac{1}{\sigma}(\mathcal E_k-\mathcal E_{k+1})=0,\quad \forall k\ge k_1.
\]
This gives rise to
\begin{eqnarray*}
 \sum_{k\ge 0}\Big(\|\Delta x^{k+1}\|
+\|\Delta y^{k+1}\|
+\|\Delta z^{k+1}\|\Big)
\le 
\sum_{k=1}^{k_1} \Big(\|\Delta x^{k}\|
+\|\Delta y^{k}\|
+\|\Delta z^{k}\|\Big)<+\infty.
\end{eqnarray*}
The latter conclusion is due to the fact that the sequence is bounded.

{\it Case 2.} The error sequence $\mathcal E_k=\mathcal R_k-\mathcal R_{\infty}>0$ for all $k\ge k_0$. 
Then by (\ref{eq:mdr11-r}) we have 
\begin{eqnarray}\label{s1}
|||\Delta u^{k+1}|||^2=\|\Delta x^{k+1}\|^2+\|\Delta y^{k+1}\|^2+\|\Delta z^{k+1}\|^2 \le 
 \frac{1}{\sigma}(\mathcal E_k-\mathcal E_{k+1}),\quad \forall k\ge k_0.
\end{eqnarray}
By Lemma \ref{lma-cluster2}, $\Omega$ is nonempty, compact, and connected.
Also, $\mathcal R_k$ takes on a
constant value $\mathcal R_{\infty}$ on $\Omega$. 
Since the sequence $\{\mathcal R_{k}\}_{k\ge k_0}$ is monotonically decreasing to $\mathcal R_{\infty}$,
then there exists $ k_1\ge k_0\ge 1$ such that 
\[
\mathcal R_{\infty} <\mathcal R_k< \mathcal R_{\infty}+\eta,\quad \forall k\ge k_1. 
\]
By Lemma \ref{lma-cluster2} we also have
$\lim_{k\to\infty} {\rm dist} \Big[ (x^k,y^k,z^k,y^{k-1},x^{k-1}), \Omega\Big]=0.$
Thus  there exists $k_2\ge 1$ such that 
\[
{\rm dist} \Big[ (x^k,y^k,z^k,y^{k-1},x^{k-1}), \Omega\Big]<\epsilon, \quad  \forall k\ge k_2. 
\]
Choose $\tilde k =\max\{ k_1,k_2,3 \}$
 then $(x^k,y^k,z^k,y^{k-1},x^{k-1})\in\mathcal S$
for $k\ge \tilde k$, where $\mathcal S$ defined in (\ref{eq:intersec}).
This follows that
\begin{eqnarray}\label{KL}
\psi'( \mathcal E_k)\cdot {\rm dist}\Big(0, 
\partial \mathcal R_k \Big)
 \ge1.
\end{eqnarray}
Since $\psi$ is concave, we have 
$
\psi(\mathcal E_{k}) -\psi(\mathcal E_{k+1})\ge 
 \psi'(\mathcal E_{k}) (\mathcal E_{k} - \mathcal E_{k+1}).
$
By this, together with  (\ref{s1}) and (\ref{KL}) we then obtain
\begin{eqnarray*}
|||\Delta u^{k+1}|||^2 &\le& 
\psi'( \mathcal E_k) |||\Delta u^{k+1}|||^2 \cdot {\rm dist}( 0, \partial \mathcal R_k)\\
&\le& 
\frac{1}{\sigma} \psi'( \mathcal E_k)  (\mathcal E_{k} - \mathcal E_{k+1}) \cdot {\rm dist}(0, \partial \mathcal R_k ) \\
&\le& \frac{1}{\sigma} \Big(\psi(\mathcal E_{k}) -\psi(\mathcal E_{k+1})\Big) 
 \cdot {\rm dist}(0, \partial \mathcal R_k ).
\end{eqnarray*}
By the arithmetic mean-geometric mean inequality for any $\gamma>0$ we  have
\[
|||\Delta u^{k+1}|||
\le \frac{\gamma}{2\sigma}  \Big(\psi(\mathcal E_{k}) -\psi(\mathcal E_{k+1})\Big) 
+\frac{1}{2\gamma}{\rm dist}(0, \partial \mathcal R_k ).
\]
This follows that 
\begin{eqnarray}\label{main}
\|\Delta x^{k+1}\|
+\|\Delta y^{k+1}\|
+\|\Delta z^{k+1}\|
\le \frac{\sqrt{3}\gamma}{2\sigma}  \Big(\psi(\mathcal E_{k}) -\psi(\mathcal E_{k+1})\Big) 
+\frac{\sqrt{3}}{2\gamma}{\rm dist}(0, \partial \mathcal R_k ).
\end{eqnarray}
By Lemma \ref{lma-D} we then obtain
\begin{eqnarray}\label{import1}
\|\Delta x^{k+1}\|
+\|\Delta y^{k+1}\|
+\|\Delta z^{k+1}\|
\le 
\frac{\sqrt{3}\gamma}{2\sigma}  \Big(\psi(\mathcal E_{k}) -\psi(\mathcal E_{k+1})\Big) 
+
\frac{\sqrt{3}\tilde \rho}{2\gamma}\Big(\|\Delta x^{k}\|
+\|\Delta y^{k}\|
+\|\Delta z^{k}\|\Big).
\end{eqnarray}
Exploit the identity
$
\sum_{k=\underline k}^K\|\Delta x^{k}\| =
\sum_{k=\underline k}^{K}
\|\Delta x^{k+1}\|
+\|\Delta x^{\underline k}\|-\|\Delta x^K\|,
$and choose $\gamma>0$ large enough such that $1>\sqrt{3}\tilde \rho/2\gamma$, and 
let $\delta_0=1-\frac{\sqrt{3}\tilde \rho}{2\gamma}$. 
Summing up (\ref{import1}) from $k=\underline k\ge \tilde k $ to $K\ge \underline k$ gives
\begin{eqnarray*}
&\sum_{k=\underline k}^K\|\Delta x^{k+1}\|
+\|\Delta y^{k+1}\|
+\|\Delta z^{k+1}\|
\le 
\frac{\sqrt{3}\gamma}{2\sigma\delta_0} 
\Big(\psi(\mathcal E_{\underline k}) -\psi(\mathcal E_{K+1})\Big) 
&\\
&+\frac{\sqrt{3}\tilde \rho}{2\gamma\delta_0}
\Big(\|\Delta x^{\underline k}\|
+\|\Delta y^{\underline k}\|
+\|\Delta z^{\underline k}\|\Big)
-\frac{\sqrt{3}\tilde \rho}{2\gamma\delta_0}
\Big(\|\Delta x^{K}\|
+\|\Delta y^{K}\|
+\|\Delta z^{K}\|\Big).&
\end{eqnarray*}
Recall that $\mathcal E_k$ is monotonically decreasing and
$\psi(\mathcal E_k)\ge \psi(\mathcal E_{k+1})>0$  hence%
\begin{eqnarray*}
&\sum_{k=\underline k}^K\|\Delta x^{k+1}\|
+\|\Delta y^{k+1}\|
+\|\Delta z^{k+1}\|
\le 
\frac{\sqrt{3}\gamma}{2\sigma\delta_0} 
\psi(\mathcal E_{\underline k})+\frac{\sqrt{3}\tilde \rho}{2\gamma\delta_0}
\big(\|\Delta x^{\underline k}\|
+\|\Delta y^{\underline k}\|
+\|\Delta z^{\underline k}\|\big).
\end{eqnarray*}\label{finitelength}
The right hand side of this inequality is bounded for any $K\ge \underline k$. We let $K\to\infty$ to obtain
\begin{eqnarray}\label{lambda0}
\sum_{k\ge\underline k}\|\Delta x^{k+1}\|
+\|\Delta y^{k+1}\|
+\|\Delta z^{k+1}\|
\le 
\frac{\sqrt{3}\gamma}{2\sigma\delta_0} 
\psi(\mathcal E_{\underline k})+\frac{\sqrt{3}\tilde \rho}{2\gamma\delta_0}
\Big(\|\Delta x^{\underline k}\|
+\|\Delta y^{\underline k}\|
+\|\Delta z^{\underline k}\|\Big).
\end{eqnarray}
Since $\{(x^k,y^k,z^k)\}_{k\ge 0}$ is a bounded sequence, then 
for any $\underline k\in\mathbb Z_+$ we clearly have
\begin{eqnarray}\label{lambda}
\lambda({\underline k}):=\sum_{k=1}^{\underline k}
 \|\Delta x^{k}\|
+\|\Delta y^{k}\|
+\|\Delta z^{k}\|<+\infty.
\end{eqnarray}
Thus,  by combining (\ref{lambda0}) and (\ref{lambda}) 
we conclude that $\displaystyle{\sum_{k\ge 0} \|\Delta x^{k+1}\|
+\|\Delta y^{k+1}\|
+\|\Delta z^{k+1}\|}$ is finite.

Note that for any $p,q,K\in\mathbb Z_+$ where $q\ge p>0$ we have 
\begin{eqnarray*}
|||u^q-u^p||| &=&||| \sum_{k=p}^{q-1} \Delta u^{k+1}|||
\le \sum_{k=p}^{q-1} |||\Delta u^{k+1}|||\\
&\le&   \sum_{k=p}^{q-1}
\Big(\| \Delta x^{k+1}\|
+ \| \Delta y^{k+1}\|
+ \| \Delta z^{k+1}\|\Big)\\
&\le&\sum_{k\ge 0} \|\Delta x^{k+1}\|
+\|\Delta y^{k+1}\|
+\|\Delta z^{k+1}\|\\
&<& \infty.
\end{eqnarray*}
This implies that $\{u^k\}_{k\ge 0}=\{(x^k,y^k,z^k)\}_{k\ge 0}$ is a Cauchy sequence
and converges. Moreover, by Lemma \ref{thm-cluster}, it converges to a stationary point.  $\blacksquare$
 \end{proof}

\begin{remark}
Theorem \ref{conv} gives rise to the fact that the limit point set $\omega(\{(x^k,y^k,z^k)\}_{k\ge 0})$ is a singleton. 
Let's denote by $(x^{\infty}, y^{\infty}, z^{\infty})$
the unique limit point of the sequence ${(x^k,y^k,z^k)}_{k\ge 0}$.
\end{remark}

\begin{theorem}\label{FR}
{\bf (Convergence rate of  $\mathcal R_k$)} 
Suppose that  Assumption \ref{assumptionA} and  \ref{assumptionB} (ii) hold,
and $\mathcal R$ satisfies the K{\L}  property at 
$v^{\infty}:=(x^{\infty}, y^{\infty}, z^{\infty},y^{\infty}, z^{\infty})$. That is,
there exists an exponent $\theta\in[0,1)$, $C_L>0$,  and $\epsilon>0$ such that 
for all $v:=(x,y,z,y',z')$ where ${\rm dist}(v,v^{\infty})<\epsilon$ it holds
\begin{eqnarray}\label{eq:L}
|\mathcal R(v)-\mathcal R(v^{\infty)})|^{\theta}\le C_L{\rm dist}(0,\partial \mathcal R(v)).
\end{eqnarray}
Denote 
$\mathcal E_k:=\mathcal R_k-\mathcal R_{\infty}$,
where $\mathcal R_{\infty}:=\mathcal R(v^{\infty})=\lim_{k\to\infty} \mathcal R_k$. 
There exists $K\ge 1$  such that 
\begin{eqnarray}\label{Eformula}
\bar\alpha \mathcal E_{k}^{2\theta}\le \mathcal E_{k-1}-\mathcal E_{k},\quad \forall k\ge K
\end{eqnarray}
where $\bar\alpha>0$. Moreover, 
\begin{itemize}
\item [(a)] if $\theta=0$, then $\mathcal E_k$ converges to zero in a finite number of iterations.
\item [(b)] if $\theta\in (0,1/2]$, then for all $k\ge K$ it holds

\begin{eqnarray}
\mathcal E_k\le
 \displaystyle{\frac{ \max\{\mathcal E_i:1\le i\le K\}}{(1+\bar\alpha \mathcal E_K^{2\theta-1})^{k-K+1} }},
\end{eqnarray}

 \item [(c)] if $\theta\in (1/2,1)$ then there is a  $\mu>0$ such that for all $k\ge K$ it holds
    \[
 \mathcal E_{k}\le \Big(\frac{1}{\mu  (k-K)+ \mathcal E_K^{1-2\theta}} \Big)^{\frac{1}{2\theta-1}},
  \quad \forall k\ge K.
   \]
\end{itemize}
\end{theorem}
\begin{proof}
By (\ref{nD}) for any $k\ge 1$ we have 
 \begin{eqnarray}\label{ss}
\frac{1}{3\tilde \rho^2}|||s^k|||^2 \le   \|\Delta x^k\|^2+\|\Delta y^k\|^2+ \|\Delta z^k\|^2.
 \end{eqnarray}
By (\ref{eq:mdr11-r}) there exists a $k_0\ge 1$ such that for any $k\ge k_0$ we have 
 \begin{eqnarray}\label{sss}
\|\Delta x^{k}\|^2+\|\Delta y^{k}\|^2+\|\Delta z^{k}\|^2
\le \frac{1}{\sigma} (\mathcal E_{k-1}-\mathcal E_{k}).
 \end{eqnarray}
Combining (\ref{ss}) and (\ref{sss}) leads to
\begin{eqnarray}\label{ll}
\frac{1}{3\tilde \rho^2}|||s^k|||^2 \le \frac{1}{\sigma} (\mathcal E_{k-1}-\mathcal E_{k}).
\end{eqnarray}

Since $\mathcal R$ satisfies the \L ojasiewicz property at $v^{\infty}$,
$v^k\to v^{\infty}$,  $\mathcal R_k$ monotonically decreasing, and $\mathcal R_k\to \mathcal R_{\infty}$ as $k\to \infty$, 
then there exist an $K\ge k_0$, $\epsilon>0$, $\theta \in [0,1)$,  and $C_L>0$ such that   for all $k\ge K$
${\rm dist}(v^k,v^{\infty})<\epsilon$ and it holds
$|\mathcal R_k-\mathcal R_{\infty}|^{\theta}\le C_L\; {\rm dist}(0,\partial \mathcal R_k)
$.
Hence we have 
$\mathcal E_k^{\theta}\le C_{L}|||s^k|||$,
and therefore 
\begin{eqnarray*}
\mathcal E_k^{2\theta}\le C_L^2 |||s^k|||^2, \quad \forall k\ge K
\end{eqnarray*}
for some $s^k\in\partial\mathcal R_k$.
This, together with  (\ref{ll}) yields 
\[
\frac{\sigma}{3C_L^2\tilde \rho^2} \mathcal E_k^{2\theta} \le \mathcal E_{k-1}-\mathcal E_{k}.
\]
Setting $\bar\alpha=\sigma/{3C_L^2\tilde \rho^2}>0$, we obtain (\ref{Eformula}). 

({\bf i}) Let $\theta=0$. If $\mathcal E_k>0$ for $k\ge K$ we would have $\bar\alpha \le \mathcal E_{k-1}-\mathcal E_k$.
As $k$ approaches  infinity, the right hand side approaches  zero, then $0<\bar\alpha\le 0$, which leads to a contradiction. 
Hence  $\mathcal E_k$ must be equal to zero for  $k\ge K$.
Hence, there is a $\tilde k\le K$ such that $\mathcal E_k=0$ for all $k\ge\tilde k$.

({\bf ii}) If $\theta\in (0,\frac 12]$, then $2\theta-1<0$.  Let $k\ge K+1$ be fixed.
$\{\mathcal E_i\}_{i\ge K}$ is monotonically decreasing, $\mathcal E_i\le \mathcal E_K$ for $i= K+1, K+2, \dots, k$ and 
 \[
\bar\alpha \mathcal E_K^{2\theta-1} \mathcal E_k\le  \mathcal E_{k-1}-\mathcal E_{k},\quad \forall k>K+1.
 \]
 We rearrange this to obtain
   \[
\mathcal E_k \le 
\frac{ \mathcal E_{k-1}}{1+\bar\alpha \mathcal E_K^{2\theta-1}}
\le\frac{ \mathcal E_{k-2}}{(1+\bar\alpha \mathcal E_K^{2\theta-1})^2}
\le \dots\le
 \frac{\mathcal E_{K}}{(1+\bar\alpha \mathcal E_{k_0}^{2\theta-1})^{k-K}}. 
 \]
Hence
\[
\mathcal E_k
 \le  \frac{ \max\{\mathcal E_i:0\le i\le K\}}{(1+\bar\alpha \mathcal E_K^{2\theta-1})^{k-K}},
 \quad k\ge K.
\]
 
({\bf iii}) Let $\theta\in(1/2,1)$. Rearrange (\ref{Eformula}) to obtain 
 \begin{eqnarray}\label{EformulaR}
\bar\alpha \le (\mathcal E_{k-1}-\mathcal E_{k}) \mathcal E_{k}^{-2\theta},\quad \forall k\ge K
\end{eqnarray}
We let  $h:\mathbb R_+\to \mathbb R$ defined by $h(s)=s^{-2\theta}$ for $s\in\mathbb R_+$. 
Clearly, $h$ is monotonically decreasing  ($h'(s)=-2\theta s^{-(1+2\theta)}<0$). Since $\mathcal E_k\le \mathcal E_{k-1}$ for all $k\ge K$ then  $h(\mathcal E_{k-1}) \le h(\mathcal E_{k})$ for all $k\ge K$ as $\mathcal E_k$ is monotonically decreasing.
We consider two cases. 
First, let $r_0\in (1,+\infty)$ such that 
\[
h(\mathcal E_{k}) \le r_0 h(\mathcal E_{k-1}),\quad \forall k>K.
\] 
Hence, by (\ref{EformulaR}) we obtain
 \begin{eqnarray*}
\bar\alpha \le r_0 (\mathcal E_{k-1}-\mathcal E_{k}) h(\mathcal E_{k-1}) 
  &\le& r_0 h(\mathcal E_{k-1})  \int^{\mathcal E_{k-1}}_{\mathcal E_{k}} 1ds \\
  &\le & r_0 \int^{\mathcal E_{k-1}}_{\mathcal E_{k}} h(s) ds\\
  &= & r_0 \int^{\mathcal E_{k-1}}_{\mathcal E_{k}} s^{-2\theta} ds\\
  &=&   \frac{r_0}{1-2\theta}[\mathcal E_{k-1} ^{1-2\theta} - \mathcal E_{k}^{1-2\theta}].
 \end{eqnarray*}
Note that $1-2\theta<0$, so rearrange to get  
  \[
  0<  \frac{\bar\alpha(2\theta-1)}{r_0} \le  \mathcal E_{k}^{1-2\theta}-\mathcal E_{k-1} ^{1-2\theta}. 
  \]
 Setting $\hat\mu=\frac{\bar\alpha(2\theta-1)}{r_0}>0$  and $\nu:={1-2\theta}<0 $ one then can obtain
  \begin{eqnarray}\label{F1}
  0<\hat\mu<\mathcal E_{k}^{\nu}-\mathcal E_{k-1}^{\nu},\quad \forall k> K.
  \end{eqnarray}

 Next, we consider the case where $h(\mathcal E_{k}) \ge r_0 h(\mathcal E_{k-1})$,. 
  This yields $\mathcal E_{k}^{-2\theta} \ge r_0 \mathcal E_{k-1}^{-2\theta}$. 
Rearranging this gives
$r_0^{-1}\mathcal E_{k-1}^{2\theta} \ge \mathcal E_{k}^{2\theta}$, which by raising both sides 
to the power $1/2\theta$ and setting $q:=r_0^{-\frac{1}{2\theta}}\in(0,1)$
leads to 
   \[
q\mathcal E_{k-1}\ge \mathcal E_{k}.
  \]
Since $\nu=1-2\theta<0$, $q^{\nu} \mathcal E_{k-1}^{\nu}\le \mathcal E_{k}^{\nu}$, 
which follows that
\[
   (q^{\nu}-1) \mathcal E_{k-1}^{\nu}\le \mathcal E_{k}^{\nu} -\mathcal E_{k-1}^{\nu}.
\]
By the fact that $q^{\nu}-1>0$ and $\mathcal E_p\to 0^+$ as  $p\to\infty$, there exists $\bar\mu$ such that 
  $(q^{\nu}-1) \mathcal E_{k-1}^{\nu}> \bar \mu$ for all $k> K$. Therefore we obtain
  \begin{eqnarray}\label{F2}
0<\bar \mu \le \mathcal E_{k}^{\nu} -\mathcal E_{k-1}^{\nu}.
   \end{eqnarray}

Choose $\mu=\min\{\hat \mu,\bar \mu\}>0$, one can combine (\ref{F1}) and (\ref{F2}) to obtain  
\[
0<\mu \le \mathcal E_{k}^{\nu} -\mathcal E_{k-1}^{\nu},\quad \forall k> K.
\]
Summing this inequality from $K+1$ to some $k\ge K+1$ gives
\[
   \mu  (k-K)+ \mathcal E_K^{\nu} \le \mathcal E_{k}^{\nu} .
\]
Hence   
   \[
\displaystyle{ \mathcal E_{k}\le ( \mu  (k-K)+ \mathcal E_K^{\nu} )^{1/\nu}
  =( \mu  (k-K)+ \mathcal E_K^{1-2\theta} )^{1/ (1-2\theta)}.}
   \]
   This concludes the proof. $\blacksquare$
   \end{proof}

 \begin{theorem}  \label{seqrate}
{\bf (Convergence rate of sequence)}
Suppose that the Assumptions \ref{assumptionA} and \ref{assumptionB} (ii) hold,
and $u^{\infty}:=(x^{\infty}, y^{\infty}, z^{\infty})$ is the unique limit point of the sequence 
$\{(x^k,y^k,z^k)\}_{k\ge 0}$ generated by the PL-ADMM algorithm. 
\begin{itemize}
\item [(a)]
If $\mathcal R$ satisfies the {K\L} property at $v^{\infty}:=(x^{\infty}, y^{\infty}, z^{\infty},y^{\infty}, z^{\infty})$
then there exists a $K\ge 1$ such that for all $k\ge K$ we have
\begin{eqnarray}\label{pp}
|||u^k-u^{\infty}|||\le C \max\{ \psi(\mathcal E_{k}),  \sqrt{\mathcal E_{k-1}} \},
\end{eqnarray}
where $C>0$ constant, $\mathcal E_k:=\mathcal R_k-\mathcal R_{\infty}$,
$\mathcal R_{\infty}:=\mathcal R(v^{\infty})=\lim_{k\to\infty} \mathcal R_k$,
$\psi\in\Psi_{\eta}$ with $\eta>0$ denotes a desingularizing function. 
\item [(b)] Moreover, if \[
\psi:[0,\eta)\to[0,+\infty),\;\; \psi(s)=s^{1-\theta},\quad{\rm where}\quad \theta\in[0,1)
\]
then the following rates hold
\begin{itemize}
\item [(i)] If $\theta=0$,  then $u^k$ converges to $u^{\infty}$ in a finite number of iterations. 
\item [(ii)] If $\theta\in (0,1/2)$, then for all $k\ge K$ it holds
 \[
|||u^k-u^{\infty}||| \le 
\frac{ 
\max\{\sqrt{\mathcal E_i}:1\le i\le K\}
}{\sqrt{(1+\bar\alpha \mathcal E_K^{2\theta-1})^{k-K+1}}},
\]   
where $\tilde\alpha =\sigma/3\tilde\rho^2$.
\item [(iii)] if $\theta\in(1/2,1)$ then
  \[
  |||u^k-u^{\infty}||| 
\le\Big(\frac{1}{\mu (k-K+1)+{\mathcal E_K}^{1-2\theta}} \Big)^{\frac{1-\theta}{2\theta-1}},\quad \forall k\ge K. 
  \]     
\end{itemize}
\end{itemize}
\end{theorem}

\begin{proof}
Part (a). Let $k_0\ge 1$ such that $\{\mathcal E_k\}_{k\ge k_0}$ is monotonically decreasing.
By (\ref{eq:mdr11-r}) 
for all   $k\ge k_0+1$  it holds
\begin{eqnarray}\label{sasa}
\|\Delta x^k\|+\|\Delta y^k\|+\|\Delta z^k\|\le
\frac{\sqrt{3}}{\sqrt{\sigma}} \sqrt{\mathcal E_{k-1}-\mathcal E_{k}}
 \le  \frac{\sqrt{3}}{\sqrt{\sigma}} \sqrt{\mathcal E_{k-1}}.
\end{eqnarray}
By this, and the fact that $\mathcal R_k$ converges to $\mathcal R_{\infty}$, $\lim_{k\to\infty} v^k=v^{\infty}$,
and $\mathcal R$ satisfies the {K\L} property at $v^{\infty}$, there exists $\epsilon>0$, $\eta>0$ 
and $\psi\in\Psi_{\eta}$, and $K\ge k_0+1$ such that for all $k\ge K$, we have
${\rm dist}(v^k, v^{\infty}) <\epsilon$ and $\mathcal R_{\infty} <\mathcal R_k< \mathcal R_{\infty}+\eta,
$ and the following {K\L} property holds
\begin{eqnarray}\label{KLlast}
\psi'\big(\mathcal E_k\big)\cdot {\rm dist}\big(0, \partial \mathcal R_k\big)\ge 1.
\end{eqnarray}
By the concavity of $\psi$ and (\ref{s1})  we then obtain
\[
|||\Delta u^{k+1}|||^2 \le \frac{1}{\sigma} \Big(\psi(\mathcal E_{k}) -\psi(\mathcal E_{k+1})\Big) 
 \cdot {\rm dist}(0, \partial \mathcal R_k ).
\]
By the arithmetic mean-geometric mean inequality for any $\gamma>0$ we  have
\begin{eqnarray*}
\|\Delta x^{k+1}\|
+\|\Delta y^{k+1}\|
+\|\Delta z^{k+1}\|
\le \frac{\sqrt{3}\gamma}{2\sigma}  \Big(\psi(\mathcal E_{k}) -\psi(\mathcal E_{k+1})\Big) 
+\frac{\sqrt{3}}{2\gamma}{\rm dist}(0, \partial \mathcal R_k ).
\end{eqnarray*}
Using Lemma \ref{lma-D} gives
\begin{eqnarray}\label{import}
\|\Delta x^{k+1}\|
+\|\Delta y^{k+1}\|
+\|\Delta z^{k+1}\|
\le 
\frac{\sqrt{3}\gamma}{2\sigma}  \psi(\mathcal E_{k}) 
+
\frac{\sqrt{3}\tilde \rho}{2\gamma}\Big(\|\Delta x^{k}\|
+\|\Delta y^{k}\|
+\|\Delta z^{k}\|\Big).
\end{eqnarray}
Let  $\gamma>0$ large enough such that $1>\sqrt{3}\tilde \rho/2\gamma$. 
Denote $\delta_0:=1-\frac{\sqrt{3}\tilde \rho}{2\gamma}$, then sum up the latter inequality over $k\ge K$
to get %
\begin{eqnarray*}
&\sum_{k\ge K}\|\Delta x^{k+1}\|
+\|\Delta y^{k+1}\|
+\|\Delta z^{k+1}\|
\le 
\frac{\sqrt{3}\gamma}{2\sigma\delta_0} 
\psi(\mathcal E_K)+\frac{\sqrt{3}\tilde \rho}{2\gamma\delta_0}
\Big(\|\Delta x^K\|
+\|\Delta y^K\|
+\|\Delta z^K\|\Big).
\end{eqnarray*}\label{finitelength}
Hence by the triangle inequality for any $k\ge K$ it holds 
\begin{eqnarray*}
|||u^k-u^{\infty}|||&\le& \sum_{p\ge k}|||\Delta u^{p+1}|||\\
&\le &\sum_{p\ge k} \|\Delta x^{p+1}\|
+\|\Delta y^{p+1}\|
+\|\Delta z^{p+1}\|\\
&\le&
 \frac{\sqrt{3}\gamma}{2\sigma\delta_0} \psi(\mathcal E_k)
+\frac{\sqrt{3}\tilde \rho}{2\gamma\delta_0}
\Big(\|\Delta x^{k}\|
+\|\Delta y^{k}\|
+\|\Delta z^{k}\|\Big). 
\end{eqnarray*}
Use  (\ref{sasa}) to get 
\begin{eqnarray*}
|||u^k-u^{\infty}|||&\le& 
\frac{\sqrt{3}\gamma}{2\sigma\delta_0} \psi(\mathcal E_k)
+\frac{3\tilde \rho}{2\gamma\delta_0\sqrt{\sigma}}
\sqrt{\mathcal E_{k-1}}\\
&\le&
C\max\{ 
\psi(\mathcal E_{k}),  \sqrt{\mathcal{E}_{k-1}} \},
\end{eqnarray*}
where
\[
C=\max\Big\{
\frac{\sqrt{3}\gamma}{2\sigma\delta_0},
\frac{3\tilde \rho}{2\gamma\delta_0\sqrt{\sigma}}\Big\}.
\]

Part (b). By the concavity of $\psi$ it follows that 
\begin{eqnarray}\label{SR}
|||u^k-u^{\infty}|||
\le 
C\max\{ 
\mathcal{E}_{k}^{1-\theta},  \sqrt{\mathcal{E}_{k-1}} \},\quad k\ge K. 
\end{eqnarray}

We let now $\theta\in [0,1)$ and $\psi(s)=s^{1-\theta}$, then 
$\psi'(s)=(1-\theta)s^{-\theta}$. Then  (\ref{KLlast}) yields
\[
{\mathcal E_k}^{\theta}\le {\rm dist}\big(0, \partial \mathcal R_k\big),\quad \forall k\ge K.
\]
This implies that $\mathcal R_k$ satisfies the {\L}ojasiewics (\ref{eq:L}) at 
$v^{\infty}$ for all $k\ge K$ with $C_L=1$. 

{\bf (i)} If $\theta=0$, then $\mathcal E_k\to 0$ in a finite numbers of iterations.
Hence by (\ref{SR}) $u^k$ must converge to $u^{\infty}$ in a finite numbers of iterations.

{\bf (ii) } If $\theta\in (0,1/2)$, then $\max\{\mathcal {E}_{k}^{1-\theta},  \sqrt{\mathcal E_{k-1}} \} =\sqrt{\mathcal E_{k-1}}$. By Theorem \ref{FR}(ii) 
\[
|||u^k-u^{\infty}||| \le 
\frac{ 
\max\{\sqrt{\mathcal E_i}:1\le i\le K\}
}{\sqrt{(1+\bar\alpha \mathcal E_K^{2\theta-1})^{k-K+1} }},\quad \forall k\ge K
\]   
where $\bar\alpha =\sigma/3\tilde\rho^2$.

{\bf (iii) }  If $\theta\in (1/2,1)$, then  
$\max\{\mathcal E_{k-1}^{1-\theta},  \sqrt{\mathcal E_{k-1}} \} =\mathcal{E}_{k-1}^{1-\theta}$.
By Theorem \ref{FR}(iii) we have
\[
|||u^k-u^{\infty}||| 
\le\Big(\frac{1}{\mu (k-K-1)+{\mathcal E_K}^{1-2\theta}} \Big)^{\frac{1-\theta}{2\theta-1}},\quad \forall k\ge K.
 \]     
This completes the proof. $\blacksquare$  
 \end{proof}

 \section{Concluding Remarks}\label{sec6}
In this paper, we considered the variable metric proximal linearized ADMM method
(\ref{Alg1}) and established its  convergence and convergence ratea. The algorithm solves 
a broad class of linearly constrained nonconvex and nonsmooth minimization problems of the 
form (\ref{OP}).
We proved that the convergence by showing that the PL-ADMM sequence has a finite length and it is Cauchy. 
Under the powerful Kurdyka-{\L}ojasiewicz ({K\L}) property, we established the convergence rates
 for the values and the iterates, and 
we showed that various values of K{\L}-exponent associated with the objective function 
can raise the PL-ADMM with three different convergence rates.
 More precisely, we showed that if the  ({K\L}) exponent $\theta=0$, 
 the sequence generated by LP-ADMM converges in a finite numbers of iterations.
 If $\theta\in(0,1/2]$, then the sequential rate of convergence is $cQ^{k}$ where
 $c>0$, $Q\in(0,1)$, and $k\in\mathbb N$  is the iteration number.
If  $\theta\in(1/2,1]$, then the $\mathcal O(1/k^{r})$ rate where $r=(1-\theta)/(2\theta-1)$
 is achieved.

\bibliographystyle{siam}
\bibliography{library.bib}

\end{document}